\documentclass[11pt]{article}
\usepackage{cite}
\usepackage{subfigure}
\usepackage{graphicx, color, graphpap}
\usepackage{amsmath,amssymb,amsthm}
\usepackage{ifthen}
\usepackage{mathrsfs}
\usepackage{mathtools}
\usepackage{enumitem}
\usepackage{physics}
\usepackage{pstricks}
\usepackage{lineno}
\usepackage{float}
\usepackage{algorithmic}
\usepackage{fancyhdr}
\usepackage[us,12hr]{datetime} 
\usepackage{exscale}
\usepackage{tabularx}
\usepackage{latexsym}
\usepackage{fullpage}       
\usepackage{float}
\usepackage{verbatim}
\usepackage{multirow}
\usepackage{overpic}
\usepackage{comment} 

\usepackage{booktabs}
\usepackage{makeidx}
\usepackage{algorithm,algorithmic}
\usepackage{mathrsfs}
\usepackage{kbordermatrix}

\numberwithin{equation}{section}  
\numberwithin{table}{section}
\numberwithin{algorithm}{section}
\makeindex

\usepackage{csvsimple}

\usepackage[pagebackref=true]{hyperref}
\hypersetup{
plainpages=false,       
unicode=false,          
pdftoolbar=true,        
pdfmenubar=true,        
pdffitwindow=false,     
pdfstartview={FitH},    
pdftitle={Nearest DS Matrix:
Nearest Doubly Stochastic Matrix
},    
pdfnewwindow=true,      
colorlinks=true,        
linkcolor=blue,         
citecolor=green,        
filecolor=magenta,      
urlcolor=cyan           
}

\usepackage[noabbrev]{cleveref}
\usepackage{refcheck}

\DeclareMathOperator{\argmin}{argmin}

\usepackage{mathtools}

\usepackage{refcount} 

\def\R{\mathbb{R}}
\def\Sc{\mathbb{S}}

\def\Rnt{\mathbb{R}^{n^2}}
\def\Rntp{\mathbb{R}_+^{n^2}}
\def\Rntm{\mathbb{R}^{2n-1}}

\def\Rnt{\mathbb{R}^{n^2}}

\def\Rnp{\mathbb{R}_+^n}

\def\eqref#1{{\normalfont(\ref{#1})}}

\def\eqref#1{{\normalfont(\ref{#1})}}

\newtheorem{theorem}{Theorem}[section]

\newtheorem{definition}[theorem]{Definition}

\newtheorem{prop}[theorem]{Proposition}

\newtheorem{corollary}[theorem]{Corollary}

\newtheorem{remark}[theorem]{Remark}

\newtheorem{lemma}[theorem]{Lemma}

\crefname{thm}{Theorem}{Theorems}
\Crefname{thm}{Theorem}{Theorems}
\crefname{problem}{Problem}{Theorems}
\Crefname{problem}{Problem}{Theorems}
\Crefname{assump}{Assumption}{Theorems}
\crefname{assump}{Assumption}{Theorems}
\crefname{conjecture}{Conjecture}{Theorems}
\Crefname{conjecture}{Conjecture}{Theorems}
\crefname{prop}{Proposition}{Propositions}
\Crefname{prop}{Proposition}{Propositions}
\crefname{cor}{Corollary}{Corollaries}
\Crefname{cor}{Corollary}{Corollaries}
\crefname{lem}{Lemma}{Lemmas}
\Crefname{lem}{Lemma}{Lemmas}
\theoremstyle{definition}
\crefname{defn}{definition}{definitions}
\Crefname{defn}{Definition}{Definitions}
\crefname{conj}{Conjecture}{Conjectures}
\Crefname{conj}{Conjecture}{Conjectures}
\crefname{remark}{Remark}{Remarks}
\Crefname{remark}{Remark}{Remarks}
\crefname{rmk}{Remark}{Remarks}
\Crefname{rmk}{Remark}{Remarks}
\crefname{example}{Example}{Examples}
\Crefname{example}{Example}{Examples}
\crefname{align}{}{}
\Crefname{align}{}{}
\crefname{equation}{}{}
\Crefname{equation}{}{}

\newcommand{\textdef}[1]{\textit{#1}\index{#1}}

\newcommand{\cV}{{\mathcal V} }

\newcommand{\DSp}{\textbf{DS}}

\newcommand{\cM}{{\mathcal M}}

\newcommand{\A}{{\mathcal A}}

\newcommand{\bbm}{\begin{bmatrix}}
\newcommand{\ebm}{\end{bmatrix}}
\newcommand{\bem}{\begin{pmatrix}}
\newcommand{\eem}{\end{pmatrix}}
\newcommand{\beq}{\begin{equation}}
\newcommand{\beqs}{\begin{equation*}}
\newcommand{\bet}{\begin{table}}
\newcommand{\eeq}{\end{equation}}
\newcommand{\eeqs}{\end{equation*}}
\newcommand{\beqr}{\begin{eqnarray}}

\renewcommand{\vec}{{\rm vec}}

\DeclareMathOperator{\range}{range}

\DeclareMathOperator{\dist}{dist}
\DeclareMathOperator{\ext}{ext}
\DeclareMathOperator{\kvec}{{vec}}

\DeclareMathOperator{\Blkdiag}{{Blkdiag}}

\DeclareMathOperator{\Diag}{{Diag}}

\DeclareMathOperator{\Mat}{{Mat}}

\DeclareMathOperator{\conv}{{conv}}

\newcommand{\nc}{\newcommand}
\nc{\arrow}{{\rm arrow\,}}
\nc{\Arrow}{{\rm Arrow\,}}
\nc{\BoDiag}{{\rm B^0Diag\,}}
\nc{\bodiag}{{\rm b^0diag\,}}

\nc{\Mm}{{\mathcal M}^{m} }
\nc{\Mmn}{{\mathcal M}^{mn} }
\nc{\Mnr}{{\mathcal M}_{nr} }
\nc{\Mnmr}{{\mathcal M}_{(n-1)r} }
\nc{\kwqqp}{Q{$^2$}P\,}
\nc{\kwqqps}{Q{$^2$}Ps}

\nc{\notinaho}{(X,S)\in \overline{AHO}(\A)}
\nc{\inaho}{(X,S)\in AHO(\A)}

\newcommand{\bea}{\begin{eqnarray}}%
\newcommand{\eea}{\end{eqnarray}}%
\newcommand{\beas}{\begin{eqnarray*}}%
\newcommand{\eeas}{\end{eqnarray*}}%
\newcommand{\Rmn}{\R^{m \times n}}%
\newcommand{\Rnn}{\R^{n \times n}}%
{}

\newcommand{\Hnp}[1][]{\,\mathbb{H}_+^{\ifthenelse{\equal{#1}{}}{n}{#1}}}
\newcommand{\Hn}[1][]{\,\mathbb{H}^{\ifthenelse{\equal{#1}{}}{n}{#1}}}
\newcommand{\Hk}[1][]{\,\mathbb{H}^{\ifthenelse{\equal{#1}{}}{k}{#1}}}
\newcommand{\Dn}[1][]{\,\mathbb{D}^{\ifthenelse{\equal{#1}{}}{n}{#1}}}

\begin{document}
\title{
A Semismooth Newton-Type Method for the Nearest Doubly Stochastic Matrix Problem
} 
\author{
\and
\href{https://huhao.org/}{Hao Hu}
\thanks{
Department of Combinatorics and Optimization
        Faculty of Mathematics, University of Waterloo, Waterloo, Ontario, Canada N2L 3G1; Research supported by The Natural Sciences and Engineering Research Council of Canada; \url{www.huhao.org}.
}
\and
\href{https://uwaterloo.ca/combinatorics-and-optimization/about/people/n6graham}
{Haesol Im}
\thanks{
Department of Combinatorics and Optimization
        Faculty of Mathematics, University of Waterloo, Waterloo, Ontario, Canada N2L 3G1; Research supported by The Natural Sciences and Engineering Research Council of Canada.
}
\and
\href{}{Xinxin
Li}\thanks{School of Mathematics, Jilin University, Changchun, China. 
E-mail: 
{\tt xinxinli@jlu.edu.cn}. This work was supported by the National Natural 
Science Foundation of China (No.11601183) and  Natural Science Foundation 
for Young Scientist of Jilin Province (No. 20180520212JH).}
\and
\href{http://www.math.uwaterloo.ca/~hwolkowi/}
{Henry Wolkowicz}%
        \thanks{Department of Combinatorics and Optimization
        Faculty of Mathematics, University of Waterloo, Waterloo, Ontario, Canada N2L 3G1; Research supported by The Natural Sciences and Engineering Research Council of Canada; 
\url{www.math.uwaterloo.ca/\~hwolkowi}.
}
}


\break
\date{\currenttime, \today \\
Department of Combinatorics and Optimization\\
Faculty of Mathematics, University of Waterloo, Canada. \\
Research supported by NSERC.
}
\maketitle
\tableofcontents
\listoffigures
\listoftables
\listofalgorithms

\begin{abstract}
We study a semismooth Newton-type method for the 
nearest doubly stochastic matrix problem where both differentiability
and nonsingularity of the Jacobian can fail.
The optimality conditions for this problem are
formulated as a system of strongly semismooth functions. 
We show that the so-called local error bound condition does not hold for
this system. Thus the guaranteed convergence rate of Newton-type methods is at most superlinear. By exploiting the problem structure, we construct a 
modified two step semismooth Newton method that guarantees a nonsingular
Jacobian matrix at each iteration, and that converges to the nearest doubly stochastic matrix quadratically. To the best of our knowledge, this is 
the first Newton-type method which converges $Q$-quadratically in the absence of the local error bound condition.
\end{abstract}


\medskip

\section{Introduction}
Newton's method is a powerful, popular iterative technique for solving
systems of nonlinear equations. The popularity arises from its fast 
asymptotic convergence rate. But this fast convergence requires
assumptions such as: nonsingularity of the Jacobian matrix at
the solution,  or the so-called \emph{local error bound condition},
see~\Cref{def:localerrorbnd} below, and
e.g.,~\cite{fernandez2017newton,dan2002convergence,qi1993nonsmooth,mordukhovich2006variational,clarke1990optimization}. 
These assumptions unfortunately can fail for many interesting applications. 
Recent extensions when nonsingularity fails in the differentiable case
appears in e.g.,~\cite{doi:10.1080/10556789308805545,HUESO200977} and the references therein.
In this paper, we present a two-step semismooth Newton-type algorithm 
for the nearest doubly stochastic matrix problem. 
We illustrate that it is still possible to achieve a Q-quadratic
convergence rate even if
the above assumptions fail. To our knowledge this is the
first Newton-type method to have a provable Q-quadratic convergence rate without
the local error bound condition. 
We include empirical evidence that illustrates the improved speed and
accuracy of our algorithm compared to several other methods in the
literature.

The proposed algorithm is also closely
	related to the recent developments for solving semidefinite programming
	relaxations using alternating direction method of multipliers (ADMM), 
	see~\cite{BW80,BoydParikhChuPeleatoEckstein:11}. The ADMM is recently proven to be a powerful method
	for solving facially reduced semidefinite programs.
	It is currently the most efficient technique to approximately solve the
	semidefinite relaxations of various hard combinatorial problems, see
	\cite{GHILW:20,OliveiraWolkXu:15,hu2020solving,HSW:19}. For example, the nearest doubly stochastic matrix problem can serve as a subproblem in solving certain relaxations for the quadratic assignment problem, e.g., \cite{li2020efficient,GHILW:20}. An efficient algorithm for solving this subproblem is the key to push the computational limit further. The algorithm presented in this paper is efficient and robust, which serves this purpose.

\subsection{Preliminaries}
A doubly stochastic matrix is a nonnegative square matrix $X \in
\Rnn$ whose rows and columns sum to one. Doubly stochastic
matrices have many applications for example in economics, probability
and statistics, quantum mechanics, communication theory and operation
research, e.g.,~\cite{brualdi1988some,louck1997doubly,marcus1960some}. 
The nearest doubly stochastic matrix, but with a prescribed entry, has
been studied in~\cite{bai2007computing}; it is related to the numerical 
simulation of large circuit networks.

\index{data, $\hat{X}$}
\index{$\hat{X}$, data}

Throughout this paper we assume we are
given a matrix $\hat{X} \in \R^{n\times n}$. The problem of computing its nearest doubly stochastic matrix is formally given by
\begin{equation}
\label{dsm:main}
\begin{array}{rl}
\min & \|X-\hat{X}\|^{2} \\
\text{s.t.} & Xe = e, \\
& X^Te = e, \\
& X\geq 0,
\end{array}
\end{equation}
where $\|\cdot\|$ is the Frobenius norm, and $e \in \R^{n}$ is the all-ones vector.
Here, the column sum constraints appear first. Moreover, the constraints
can be viewed within the family of \emph{network flow problems} as they
define the \emph{assignment problem} constraints, e.g.,~\cite[Chap. 7]{BT97}. 

\index{Frobenius norm, $\|X\|$}
\index{$\|X\|$, Frobenius norm}

\subsubsection{A Vectorized Formulation and Optimality Conditions}

\index{\DSp, doubly stochastic}
\index{doubly stochastic, \DSp}
\index{$e$, all-ones vector}
\index{all-ones vector, $e$}
\index{$I$, identity matrix}
\index{identity matrix, $I$}

The nearest doubly stochastic matrix problem \cref{dsm:main} is defined using
the matrix variable $X\in \Rnn$. It is often more convenient to work with vectors, and therefore we shall derive an equivalent formulation using a vector of variables $x \in \R^{n^{2}}$.

\index{Kronecker product, $\otimes$}
\index{$\otimes$, Kronecker product}

\index{all-ones vector!$e$}
\index{all-ones vector!$\bar{e}\in \R^{2n}$}
\index{$\kvec$}		
\index{$\Mat$}

Let $x = \kvec(X)\in \R^{n^2}$ denote the vector obtained by stacking the
columns of $X\in \Rnn$. Conversely, $X = \Mat(x)\in \Rnn$ is the unique
matrix such that $x = \kvec(X)$. Recall that the matrix equation $AXB =
C$ can be written as $(B^{T} \otimes A)\kvec(X) = \kvec(C)$, where 
$\otimes$ denotes the Kronecker product. Therefore,
the equality constraints in \cref{dsm:main} are equivalent to
$$
(I \otimes e^{T})x = e \text{ and } (e^{T} \otimes I)x = e. 
$$
Thus we can express the feasible region of \cref{dsm:main} in vector
form as the set $\{x \in \R^{n^{2}} :  \bar A x = \bar{e}, x\geq 0\}$, where $\bar{e} \in \R^{2n}$ is the all-ones vector and the matrix $\bar{A}$ is
\begin{equation}\label{fullA}
\bar A =\begin{bmatrix}
I \otimes e^{T} \cr e^{T} \otimes I
\end{bmatrix} = 
\begin{bmatrix}
e^T & 0 & 0 \cr
0 & \ddots & 0 \cr
0 & 0 & e^T  \cr
I & \cdots & I
\end{bmatrix} \in \R^{2n \times n^{2}}.
\end{equation}

\index{$\bar A$}

It is easy to see that one of the equality constraints is redundant.
Therefore, we discard the last row in $\bar{A}$, i.e., the $2n$-th
constraint that the last row of $X$ sums to one. 
We denote this by $A$. We observe that the
matrix with all elements $1/n$ is strictly feasible. Therefore, we now
have that the Mangasarian-Fromovitz constraint qualification, MFCQ,
holds. This further means that the set of optimal dual variables is
compact, \cite{MR34:7263,MR0489903}.

Let \textdef{$\hat{x} = \kvec(\hat{X})$}. The doubly stochastic matrix
problem \cref{dsm:main} in the vector form is given by the unique
minimum of the strictly convex minimization problem
\begin{equation}
\label{eq:mainprob}
x^* = \argmin \left\{\frac 12 \|x-\hat x\|^2 \,:\,
Ax=b, \, x\geq 0\right\},
\end{equation}
where $A \in \R^{(2n-1) \times n^{2}}$ is the first $2n-1$ rows of $\bar{A}$,
and $b \in \R^{2n-1}$ is the all-ones vector. The optimal doubly stochastic 
matrix is denoted by $X^{*} = \Mat (x^{*})$. By abuse of notation, where
needed we often use double indices for the vectors 
\textdef{$x = (x_{ij}) \in \R^{n^2}$}.\footnote{The 
perturbation function (optimal value function) is
$p^*(\epsilon) = \min \left\{\frac 12 \|x-\hat x\|^2 \,:\,
Ax=b+\epsilon, \, x\geq 0\right\}$.
Then $\partial p^*(0)=\{y\}$ is the set of optimal dual multipliers, which is
always a compact, convex set since MFCQ holds. So differentiability
holds if, and only if, it is a
singleton.}

\index{$A$, constraint matrix}
\index{constraint matrix, $A$}
\index{all-ones vector!$b\in \R^{2n-1}$}

The standard Karush-Kuhn-Tucker, KKT, optimality conditions for the 
primal-dual variables $(x,y,z)$ for~\cref{eq:mainprob} are:
\begin{equation}
\label{eq:optcondxyz_copy}
\begin{array}{rcl}
\begin{bmatrix} 
x -\hat x -A^Ty-z \cr
Ax - b  \cr
z^{T}x
\end{bmatrix} 
=0, \quad x,z\in \Rntp, y\in \Rntm.	
\end{array}
\end{equation}		
The system \eqref{eq:optcondxyz_copy} is a bilinear system of order $n^{2}$.
\Cref{thm:optcondG} below shows that we can simplify the KKT conditions
and obtain an elegant characterization of optimality. This new optimality condition
is a smaller system of size $2n-1$. However, the new system involves 
a nonsmooth (metric) projection of a given $v$ onto the nonnegative orthant,
denoted $v_+ = \argmin_x \{ \|x-v\| : x \geq 0\}$. 
(The absolute value of the
projection onto the nonpositive orthant is denoted $v_{-}$.)
Therefore we get $v = v_+ - v_-,\, v_+^Tv_-=0$.
\index{projection onto $\Rnp$, $v_{+}$}			
\index{projection onto $\R^n_-$, $v_{-}$}			
\index{$v_{+}$, projection onto $\Rnp$}			
\index{$v_{-}$, projection onto $\R^n_-$}			

\begin{theorem}
\label{thm:optcondG}
Let $\hat x\in \Rnt$ be given.
The optimal solution $x^*\in \Rnt$ for the nearest doubly stochastic problem 
\cref{eq:mainprob} exists and is unique. Moreover,
$x^*\in \Rnt$ solves \cref{eq:mainprob} if, and only if, 
\begin{equation}
\label{eq:optcondG}
x^* = (\hat x+A^Ty^*)_+, \quad
F(y^*):=A(\hat x+A^Ty^*)_+ -b = 0,\, \text{ for some  } y^* \in \Rntm.
\end{equation}
\end{theorem}
\begin{proof}
The existence and uniqueness of the optimum	$x$ follows since \cref{eq:mainprob} is a projection onto a closed convex set.
The Lagrangian dual of \cref{eq:mainprob} is
$$
\max_{z\geq 0,y} \min_x L(x,y,z) = \frac 12 \|x-\hat x\|^2 -y^T(Ax-b) -z^T x.
$$
For nonnegative vectors $z,x\geq 0$, the optimality is characterized by the KKT conditions \eqref{eq:optcondxyz_copy}, i.e.,~from
dual feasibility ($\nabla_x L(x,y,z) = (x-\hat x) -A^Ty -z = 0$), 
primal feasibility, complementary slackness, respectively.
We get
\begin{equation}
\label{eq:optcondplusminus}
0\leq x = (\hat x + A^Ty)_+ - (\hat x+A^Ty)_- + z, \,z\geq 0, \, Ax=b, \,  
z_ix_i=0, \forall i.
\end{equation}
This implies
\[
x = (\hat x+A^Ty)_+, \, z = (\hat{x}+A^{T}y)_{-}.
\]
\end{proof}

It follows from \Cref{thm:optcondG} that if
$F(y) = 0$, then $x = (\hat{x}+A^{T}y)_{+}$ is the optimal primal point
and $z = (\hat{x}+A^{T}y)_{-}$ is an optimal dual vector for
\cref{eq:mainprob}. We note that this characterization is well-known, and
it can also be derived for more general results for finite dimensional problems
e.g.,~\cite{smw2,homwolkA:04}, and for infinite dimensional
problems, see e.g.~\cite{MiSmSwWa:85,BoWo:86,deutsch1997dual,BoLe1:92}.
In \cite{QiSun:06}, this reformulation strategy is used for the nearest correlation matrix problem. They also prove that the obtained semismooth system has a nonsingular Jacobian at the optimum and leads to a very competitive algorithm.
This is in contrast to our problem, where the generalized
Jacobian at the optimum can contain many highly singular matrices.

\begin{remark}
Our problem is a special case of the linearly constrained linear least squares
problem, e.g.~\cite{MR1349828}, that is itself a special case of 
quadratic programming, e.g.~\cite{MR95d:90001}. These problems lie
within the class of linear complementarity problems,
e.g.,~\cite{MR3396730}.

In contrast to our dual type algorithm that
applies a Newton-type method to the optimality conditions, 
the approaches in the literature include:
\begin{enumerate}
\item
active set methods, e.g.,~\cite{MR3642293};
\item
quadratic cost network flow problems, e.g.,~\cite{MR1183260,MR1117772};
\item
path following, interior point methods, e.g.~\cite{MR95d:90001}, 
that also use a Newton method applied to perturbed optimality
conditions;
\item
classical Lemke and Wolfe type methods, e.g.,~\cite{MR3396730};
\item
splitting methods such as ADMM, e.g.~\cite{GHILW:20}.
\end{enumerate}
\end{remark}

\subsubsection{Semismooth Newton Methods}

\index{$D_{F}$, set of differentiable point of $F$}			
\index{set of differentiable point of $F$, $D_{F}$}		
\index{$F^\prime(y)$, Jacobian of $F$ at $y$}	
\index{Jacobian of $F$ at $y$, $F^\prime(y)$}	
\index{$\partial  F(y)$, generalized Jacobian at $y$}
\index{generalized Jacobian, $\partial  F(y)$}

In this paper we solve the nearest matrix problem by applying a semismooth
Newton method to the nonsmooth optimality conditions 
of~\cref{eq:mainprob} in the form $F(y)=0$. We now present the
preliminaries for semismooth Newton methods.

Suppose $F:\R^{s} \rightarrow \R^{t}$ is locally Lipschitzian. According to Rademacher's Theorem \cite{Rademacher}, $F$ is Frech\'et differentiable almost everywhere. Denote by $D_{F}$ the set of points at which $F$ is differentiable. Let $F^\prime(y)$ be the usual Jacobian matrix at $y \in D_{F}$. The generalized Jacobian $\partial F(y)$ of $F$ at $y$ in the sense of Clarke \cite{clarke1990optimization} is
\begin{equation}
\partial  F(y) := \conv \left\{ \lim_{\substack{y_{i} \rightarrow y \\ y_{i} \in D_{F}}} F'(y_{i}) \right\}.
\end{equation}
\index{nonsingular!$\partial F(y)$}			
The generalized Jacobian $\partial F(y)$ is said to be \emph{nonsingular}, if every $V \in \partial F(y)$ is nonsingular. 
The Lipschitz continuous function $F$ is \textdef{semismooth} at $y$, if $F$ is directionally differentiable 
at $y$ and 
$$
\|F(y+d)-F(y)-Vd\| = {\scriptstyle\mathcal{O}}(\|d\|), \ \forall \; V \in \partial F(y+d) \text{ and } d\rightarrow 0 .
$$
Moreover, $F$ is \textdef{strongly semismooth} at $y$, if $F$ is semismooth at $y$ and
$$
\|F(y+d)-F(y)-Vd\| = {\scriptstyle\mathcal{O}}(\|d\|^{2}) \quad  \forall \; V \in \partial F(y+d) \text{ and } d\rightarrow 0.
$$
We note that the projection operator $v_+$ in our optimality
conditions~\cref{eq:optcondG} is a special case of a metric
projection operator and is strongly semismooth
e.g.,~\cite{sun2002semismooth,chen2003analysis}. 

Now let $y^{0}$ be a given initial point.
If $\partial F(y)$ is nonsingular, the semismooth Newton method for
solving equation $F(y) = 0$ is defined by the iterations
\index{$y^{0}$, initial point}
\index{initial point, $y^{0}$}
\begin{equation}\label{semiNM}
\boxed{	y^{k+1} = y^{k} - V_{k}^{-1}F(y^{k}),
\text{ with } V_{k} \in \partial F(y^{k}) .}
\end{equation}
A sequence $\{y^{k}\}$, is said to
\textdef{converge Q-quadratically} to $y^{*}$, if $y^k\to y^*$ and
\[
\limsup_{k\rightarrow \infty}\frac{\norm{y^{k+1}-y^{*}}}{\norm{y^{k}-y^{*}}^{2}} < M, 
\text{  for some positive constant  } M>0.
\]

The following local convergence result for the semismooth Newton method
as applied to a semismooth function $F$ is due to \cite{qi1993nonsmooth}.
\begin{theorem}\cite{qi1993nonsmooth}\label{semiNMconv}
Let $F(y^{*}) = 0$ and let
$\partial F(y^{*})$ be nonsingular. 
If $F$ is (strongly) semismooth at $y^{*}$,
then the semismooth Newton method
\cref{semiNM} is (Q-quadratically) convergent in a neighborhood of
$y^{*}$. 
\end{theorem}

The nonsingularity assumption for $\partial F(y^{*})$ can be a restrictive
assumption for the convergence of semismooth (and smooth) Newton methods. This
condition is not satisfied by many applications, including our nearest
doubly stochastic matrix problem. In these cases, regularization such 
as the Levenberg-Marquardt method (LMM) could be used to obtain the
nonsingularity. If $F$ is differentiable and the local error
bound condition is satisfied, see~\Cref{def:localerrorbnd} below,
then the LMM approach achieves quadratic convergence, 
see~\cite{levenberg1944method,marquardt1963algorithm,yamashita2001rate,fan2005quadratic}.
Note that the local error bound condition does not hold for the nearest doubly stochastic matrix problem, see \Cref{leb}.




\subsection{Contributions}
\begin{enumerate}
\item 
We present a modified two-step semismooth Newton method that exploits
the special network structure of our nearest matrix problem.
\item
At each iterate $y$, the first step finds a point (vertex) $y^\prime$
in the same equivalence class so that we can guarantee that the matrix 
chosen from the generalized Jacobian is nonsingular. Thus a regular
Newton step can be taken.
\item
This two-step method converges Q-quadratically for
the nearest doubly stochastic matrix problem. This is done in the
absence of the local error bound condition. The problem structure allows for Q-quadratic convergence to
the solution. The main idea of our algorithm is to partition the search 
space into equivalence classes so that the difficulty of singular 
generalized Jacobians can be avoided. 
\item 
The numerical tests show that our algorithm outperforms existing methods both in speed and accuracy. The algorithm is also very robust for difficult instances.
\end{enumerate}

\section{Semismooth Newton Method for Connected $X^*$}
\label{sec:semi}

In this section, we show that the semismooth Newton method \cref{semiNM} can
be used to find a solution to the optimality conditions
\cref{nm_eq} when the bipartite graph for the
optimal primal solution $X^*=\Mat(x^*)$ of \cref{eq:mainprob} is
\emph{connected}.

\subsection{Bipartite Graphs and Connectedness}
\index{bipartite graph}
For every matrix $X\in \Rmn$ we associate a bipartite graph $G=(V,E)$
with node set divided into two $V_1,V_2$ corresponding to the
rows and columns of $X$, respectively. 
The edges $(i,j) \in E$ correspond to \emph{nonzero} entries of
$X$, i.e.,~$ij\in E \iff i \in V_1, j \in V_2, X_{ij}\neq 0$.
 The adjacency matrix of $G$ can be written as 
\begin{equation}
\label{eq:Badj}
\begin{bmatrix}
	0 & B \\
	B^T & 0
\end{bmatrix}.
\end{equation}
We call the zero-one matrix $B$ the \textdef{reduced adjacency
matrix} of the bipartite graph $G$. We call $B$ and $X$
\textdef{connected} matrices if the graph $G$ is connected. We call them 
\textdef{disconnected} otherwise.

\begin{lemma}[{\textdef{connected matrix}}, 
    {\cite[page 109]{BrualRys:91}}]\label{def:cm}
	A matrix $X\in \Rmn$ is \textdef{connected} if
	there do not exist permutation matrices $P$ and $Q$ such that
	\[
	PXQ = \begin{bmatrix}
		X^1& 0 \\
		0 & X^2
	\end{bmatrix},
	\]
	where $X^1$ is $p\times q$ satisfying 
$1\leq p+q\leq m+n-1$.\footnote{A connected matrix
is often called indecomposable in  the literature.}
\end{lemma}

\index{$\bar{R},\bar{C}$, complements of $R,C$}
\index{$X^{\bar{R},\bar{C}}$}
\index{$X^{R,C}$}
\index{reduced adjacency matrix}

Let \textdef{$N=\{1,\ldots,n\}$}.
In this paper we consider square matrices $X\in \Rnn$ so that the associated 
bipartite graph has edges $ij\in N\times N$. Let
$R, C \subseteq N$,be two subsets,
with $\bar{R}$ and $\bar{C}$ the respective complements in $N$.
Then $X\in \Rnn$ can be 
partitioned and permuted using the two subsets as
\begin{equation}\label{blk_X}
	\begin{bmatrix}
		X_{\bar{R},\bar{C}} & X_{\bar{R},C}\\
		X_{R,\bar{C}} & X_{R,C}\\
	\end{bmatrix}.
\end{equation} 
We note that $X$ is connected if both $X_{R,\bar{C}},
X_{\bar{R},C}$ are zero or empty blocks, for some pair of subsets. We emphasize 
that the diagonal blocks are \emph{not} necessarily square; and if $X$ is disconnected, then one of them can be empty and thus there are zero rows or columns.

We partition the dual variables $y$ corresponding to the column and row sum
constraints as
\index{$y = \begin{pmatrix} c\cr r\end{pmatrix} \in \R^{2n-1}$}
\begin{equation}
	\label{eq:ycr}
	y = \begin{pmatrix} c\cr r\end{pmatrix} \in \R^{2n-1}, \quad c\in \R^{n}, r\in \R^{n-1}.
\end{equation}
The structure of $A$ enables us to write the equation 
$X  =  \Mat\left(\hat x+A^Ty\right)_+$ as
\begin{equation}\label{Xrc}
	X_{ij} = \begin{cases}
		( \hat{X}_{ij} + r_{i} + c_{j} )_+ & \text{ if } i \neq n, \forall j,\\
		( \hat{X}_{ij} + c_{j} )_+ &  \text{ if }i=n, \forall j.
	\end{cases}
\end{equation}

\subsection{The Algorithm for Connected $X^*$}

Recall that the optimality conditions function
\begin{equation}
\label{nm_eq}
F(y) = A(\hat{x} + A^Ty)_+ - b = 0,
\end{equation}
is \textdef{strongly semismooth}, see
\cite{sun2002semismooth,chen2003analysis}. Our algorithm is based on
applying a Newton-type method to solve this equation.
More precisely, at each iterate $y$, we have
$x=(\hat x +A^Ty)_+$ and $z=(\hat x +A^Ty)_-$, and so we guarantee
dual feasibility and complementarity:
\[
x -(\hat x+A^Ty) -z=0, \, x^Tz=0, \quad x,z\geq 0.
\]
The Newton algorithm solves $F(y)=0$ in order to obtain
the missing linear primal feasibility $Ax=b$.

Below we provide a sufficient condition for the nonsingularity of the generalized 
Jacobian $\partial F(y)$ at a $y\in \R^{2n-1}$. 
From~\Cref{semiNMconv}, we see that this
sufficient condition then guarantees that the 
semismooth Newton method converges locally to an
optimum with a Q-quadratic convergence rate 

\index{$\cM(y)$}
\index{$\partial  F(y)$, generalized Jacobian at $y$}
\index{generalized Jacobian, $\partial  F(y)$}

In order to obtain the generalized Jacobian at 
$y\in \R^{2n-1}$, we need the following set. Recall that we 
use double indices for vectors 
$x=(x_{ij})=((\hat{x}+A^{T}y)_{ij})\in \R^{n^2}$.
\begin{equation}
\label{My}
\textdef{$\cM(y)$}:= \left\{M \in \Rnn \;|\; 
M_{ij} = 
\begin{array}{cl}
&\\
\left\{\begin{array}{cl}
1 & \text{ if } (\hat{x}+A^{T}y)_{ij}>0 \\
 \left[0,1\right]  & \text{ if } (\hat{x}+A^{T}y)_{ij}=0 \\
0 & \text{ if } (\hat{x}+A^{T}y)_{ij}<0
\end{array} \right. \\
&
\end{array} 
\right \}.
\end{equation}
Note that the \emph{minimal} $M$, elementwise, is the adjacency matrix for 
$\Mat(\hat x  + A^Ty)_+$.

The generalized Jacobian of the non-linear system $\cref{nm_eq}$ at $y$
is given by the set
\begin{equation}\label{eq:jac}
\partial  F(y) = \{A\Diag(\kvec(M))A^{T} \;|\; M \in \cM(y)\}.
\end{equation}
For example, for the case where $\hat x + A^Ty
> 0$ and $\Delta y$ small, we get
\[
F(y+\Delta y) = A(\hat x + A^T(y+\Delta y))-b = F(y) +
A\Diag(\kvec(M))A^T\Delta y, \, \Diag(\kvec(M))=I.
\]
In the general case, we replace the elements of $M$
with appropriate $M_{ij}\in [0,1]$. In our applications, we choose
$M_{ij}\in \{0,1\}$, and in fact, we choose the maximal $M$ as defined
below in~\cref{maxs}. Therefore, in our applications,
every $V\in \partial F(y)$ is a sum of rank one zero-one matrices.

\index{$\cV(M)$} 
\index{$M\in \cM(y)$}	
\index{$\hat{M}$}

The next result shows that the matrices in the generalized Jacobian have
a special structure in terms of the matrices $M$ in~\cref{My}.

\begin{prop}\label{Vstructure}
Let $y\in \R^{2n-1}$ be given and let
$M\in \cM(y)$. Then the linear transformation
\begin{equation}\label{eq:jac2}
	\cV(M) :=A\Diag(\kvec(M))A^{T} \in \partial  F(y)\subset \Sc^{2n-1}_+.
\end{equation}
Moreover, $\partial F(y)$ is a nonempty, convex compact set. And
$\partial F(y)$ is a singleton if, and only if $F$ is differentiable 
if, and only if, $\cM(y)$ is a singleton. 

Now let
$$
M \in \cM(y) \in \mathbb{R}^{n \times n}, \quad
\hat M \in \mathbb{R}^{(n-1) \times n},
$$
where the latter is formed from the first $n-1$ rows of $M$. 
Then the matrix $\cV(M)$ has the following 
structure
\begin{equation}
\label{eq:MMhatstruc}
\cV(M) = \begin{bmatrix}
\Diag(M^{T}e) & \hat{M}^{T} \\
\hat{M} & \Diag(\hat{M}e)
\end{bmatrix}.\footnote{$\cV(M)$ can be obtained by removing the last row and column of the signless Laplacian of the
adjacency matrix of the graph $G$
in~\cref{eq:Badj}, \cite{MR2312332,HESSERT2021112451}.}
\end{equation}
\end{prop}
\begin{proof}
The convexity  and compactness of the generalized Jacobian are well
known properties. The singleton property is clear from the definitions.
Note that $A\geq 0$ with no zero columns.

The expression for $\cV(M)$ follows from  the structure of $A$.	
\end{proof}

For $y \in \R^{2n-1}$, the fact that the matrix $\cV(M)$ 
in \cref{eq:jac2} is nonsingular for a given $M \in \cM(y)$, is equivalent to
linear independence of $2n-1$ columns in $A$ 
associated to a subset of the positive entries in $M$. Therefore, it is
clear that one should choose as
many elements $M_{i,j}>0$ as possible to obtain a nonsingular element in the generalized Jacobian.
In what follows, we derive an alternative characterization
that connects the nonsingularity of $\cV(M)$ and the connectedness of
$M$, see~\Cref{Vsing} below.

\begin{lemma}\label{Vsing}
Let $M \geq 0$. The matrix $\mathcal{V}(M)$ is nonsingular if,
and only if, $M$ is connected.
\end{lemma}

\begin{proof}
This result can be derived easily using~\cite[Prop. 2.15]{MR1490579}.
Translated to our framework,
The result states that a set of linearly independent columns in 
our matrix $A$ forms a basis of $\R^{2n-1}$ if, and only if, the 
associated set of arcs forms a spanning tree. These arcs correspond to
the graph of our matrices $M$.
The result follows by noting that $\mathcal{V}(M)$ is positive definite 
if and only if the bipartite graph associated with $M$ is connected
and thus it contains a spanning tree.
($M$ is obtained
removing the last row and column of the signless Laplacian of the
adjacency matrix of the graph $G$
in~\cref{eq:Badj}. Results on singularity for signless Laplacians appear
in e.g.,~\cite[Prop. 2.1]{MR2312332}, and \cite{TAM20101734} for the
reduced signless Laplacian.)

%
%
%

We now include an alternative proof for the sake of self-containment. We denote
$V=\cV(M)$.

First suppose that $M$ is disconnected. We now proof $M$ is singular.
We distinguish the following two cases.
\begin{enumerate}
	\item Now consider the special case that
	$M$ contains a zero  column or a zero row among
	the first $n-1$ rows. Then there is a zero diagonal entry, see \Cref{Vstructure}. Since $V$ is
	positive semidefinite, we conclude that $V$ is singular.
	
	For the case where the last row of $M$ is zero, we have that $M^{T}e - \hat{M}^{T}e = 0$ and thus there is an eigenvector that
	has a $0$ eigenvalue, i.e., by abuse of notation and using $e$
of different dimensions, we see that
	$$V\begin{bmatrix}
		e\\
		-e
	\end{bmatrix} = \begin{bmatrix}
		\Diag(M^{T}e)e - \hat{M}^{T}e\\
		\hat{M}e - \Diag(\hat{M}e)e\\
	\end{bmatrix} = \begin{bmatrix}
		M^{T}e - \hat{M}^{T}e\\
		\hat{M}e - \hat{M}e\\
	\end{bmatrix} = 0.$$
	\item We now assume that $M$ and $\hat{M}$ can be permuted so that they have the form
	$$M = \begin{bmatrix}
		M_{1} & 0 \\
		0 & M_{2}
	\end{bmatrix} \text{ and } \hat{M} = \begin{bmatrix}
		M_{1} & 0 \\
		0 & \hat{M}_{2}
	\end{bmatrix}.$$ Using \Cref{Vstructure}, we obtain
	$$\begin{array}{rl}
		V = \begin{bmatrix}
			\Diag(M_{1}^{T}e) & 0 & M_{1}^{T} & 0 \\
			0 &	\Diag(M_{2}^{T}e) & 0 & \hat{M}_{2}^{T} \\
			M_{1} & 0 & \Diag(M_{1}e) & 0 \\
			0 &	\hat{M}_{2} & 0 & \Diag(\hat{M}_{2}e) \\
		\end{bmatrix}.
	\end{array}$$
	But then the first block column and the third block column of $V$ are linearly dependent. Thus $V$ is singular.
\end{enumerate}
Thus we have shown that $V$ is singular in both cases.  (Note that we do not need the nonnegativity condition $s\geq 0$ in this direction.)

Conversely, assume that $V$ is singular. Then there exists a non-zero vector $w = \begin{bmatrix}
u \\ v
\end{bmatrix} \in \R^{2n-1}$, for some $u \in \mathbb{R}^{n}$, $v \in \mathbb{R}^{n-1}$ such that $w^{T}Vw = 0$. We can rewrite $w^{T}Vw$ as follows.

\begin{equation}\label{eq:Vsing}
\begin{array}{rll}
w^{T}Vw &=& u^{T}\Diag(M^{T}e)u + 2 u^{T}\hat{M}^{T}v + v^{T} \Diag(\hat{M}e)v \\
&=& \sum_{j=1}^{n}u_{j}^{2}M_{n,j} + \sum_{i=1}^{n-1}\sum_{j=1}^{n} (v_{i} + u_{j})^{2}M_{i,j} \\
&=& \langle W, M \rangle,
\end{array}
\end{equation} 
where 
$$W:=\begin{bmatrix}
(v_{1}+u_{1})^2 & \cdots & (v_{1}+u_{n})^2 \\
\vdots & \ddots & \vdots \\
(v_{n-1}+u_{1})^2 & \cdots & (v_{n-1}+u_{n})^2 \\
u_{1}^2 & \cdots & u_{n}^2 \\
\end{bmatrix} \in \R^{n\times n}.$$
Up to permutation, we can assume $v_{1},\ldots,v_{k_{1}}$ and $u_{1},\ldots,u_{k_{2}}$ are the only non-zero entries in $w$, where $k_{1},k_{2}$ are nonnegative integers. Note that $k_{1}+k_{2} > 0$ as $w \neq 0$. We distinguish the following cases based on $k_{1}$ and $k_{2}$.

\begin{enumerate}
	\item Suppose that $0 < k_{1}$ and $0< k_{2} < n$.  The matrix $W$ can be partitioned correspondingly as
	$$W = \begin{bmatrix}
		W_{1} & W_{12} \\
		W_{21} & W_{2}
	\end{bmatrix}  \in \R^{n\times n},$$ with non-trivial off-diagonal blocks $W_{12}  \in \mathbb{R}^{k_{1} \times(n-k_{2})} $ and $W_{21} \in \mathbb{R}^{(n-k_{1}) \times k_{2}}$. By assumption, $W_{12}>0$ and $W_{21}>0$ are element-wise positive. Partitioning $M$ in the same way yields $$M = \begin{bmatrix}
		M_{1} & M_{12} \\
		M_{21} & M_{2}
	\end{bmatrix}.$$
	Note that ${W\geq 0}$ and  ${M \geq 0}$. As $\langle W, M \rangle  = w^{T}Vw = 0$, this implies that the off-diagonal blocks $M_{12}  \in \mathbb{R}^{k_{1} \times(n-k_{2})}$ and $M_{21} \in \mathbb{R}^{(n-k_{1}) \times k_{2}}$ must be zero. Therefore, $M$ is a block-diagonal matrix.
	
	\item Using the same argument as above, the remaining three possibilities lead to a zero row or column in $M$. They are listed below.
	$$\begin{array}{lrl}
		k_{1} = 0, k_{2} = n &\implies& M = 0; \\
		k_{1} = 0, 0 < k_{2} < n &\implies& \text{the first $k_{2}$ columns of $M$ are zeros}; \\
		k_{1} > 0, k_{2} = 0 &\implies&  \text{the first $k_{1}$ rows of $M$ are zeros}. \\
	\end{array}$$
\end{enumerate}
This shows that $M$ is disconnected.
\end{proof}

We now provide two properties of the generalized Jacobian that show the
relationships between nonsingularity, connectedness and also
differentiability and strict complementarity.
\begin{theorem}\label{sys_sing}
Let $y\in \R^{2n-1}$, and set
	\[
	X := \Mat(\hat{x}+A^{T}y)_+, \quad
	Z := \Mat(\hat{x}+A^{T}y)_-.
	\]
Then the following holds:
\begin{enumerate}
\item
The generalized Jacobian $\partial F(y)$ is 
nonsingular if, and only if, the matrix $X$ is connected.
\item
The generalized Jacobian $\partial F(y)$ is a singleton ($F$ is
differentiable at $y$)
if, and only if, strict complementarity, $X+Z>0$  holds.
\end{enumerate}
\end{theorem}

\index{nonsingular!$\partial F(y)$}			
\index{nonsingular!$\partial F(y)$}			

\begin{proof}
\begin{enumerate}
\item
Let $M^\prime \in \cM(y)$ be such that $M_{ij}^\prime = 0$ if
$X_{ij} =0$, i.e., $M^\prime$ is the smallest elementwise.
Note that $\cV(\cdot)$ is a monotonic mapping,
i.e.,~for any $M \in \cM(y)$, we have $M^\prime \leq M$ and 
thus $\cV(M^\prime) \preceq \cV(M)$. Hence we have
\[
\begin{array}{rcll}
\partial F(y) \text{ is nonsingular } & \iff & \cV(M) 
  \text{ is nonsingular $\forall M \in \cM(y)$} & \text{ (by definition)} \\
                                      & \iff & \cV(M^\prime) 
\text{ is nonsingular for smallest $M^\prime$} \\
& \iff & M^\prime \text{ is connected} & \text{ (by \Cref{Vsing})} \\
& \iff & X \text{ is connected}, & 
\end{array}
\]
where the last equivalence follows
since $X_{ij} > 0 \iff  M_{ij}^\prime > 0, \forall ij$ for the smallest $M^\prime$
\item
From the definitions of $\cM(y)$~\cref{My} and the Jacobian in
\cref{eq:jac}, and the fact that $A\geq 0$ with no zero columns,
we conclude that $\partial F(y)$ is a singleton
(differentiability) holds if, and only if, $\cM(y)$ is a singleton. By
definition, this is equivalent to strict complementarity.
Note that if $M \in \cM(y)$, if strict complementarity holds, we have 
$$
X_{ij} = 0  \implies  Z_{ij} > 0  \implies  \Mat(\hat{x}+A^{T}y)_{ij} 
         < 0 \implies  M_{ij} = 0.
$$
(See also~\Cref{Vstructure}.)
\end{enumerate}
\end{proof}

\begin{corollary}\label{main_conv}
Suppose $F(y^{*}) = 0$ and $X^{*} = \Mat(\hat{x}+A^{T}y^{*})_{+}$ is 
connected. Then the semismooth Newton method \cref{semiNM} has
local quadratic convergence to $y^{*}$.
\end{corollary}

\Cref{sys_sing} and \Cref{main_conv} show that if differentiability
fails at the optimum,
then strict complementarity fails. This type of degeneracy is
typically tied to ill-conditioning and slow convergence.
Similarly, if the optimum is disconnected, we get problems with
singular generalized Jacobians.
This motivates the next section that deals with finding nonsingular
matrices in the generalized Jacobian.

\section{An All-inclusive Semismooth Newton Method}\label{sec:mod}

In this section we develop an algorithm that allows for the 
cases where the optimal solution $X^*$ is \emph{disconnected}. 
In this case the generalized Jacobian $\partial F(y)$ of the non-linear system \cref{nm_eq} is singular.			
Hence the iterates of the semismooth Newton method \cref{semiNM} are
not well-defined, and the convergence result in \Cref{main_conv}
is not applicable, see e.g.,~\cite{qi1993nonsmooth}.
In fact, the iterate in \cref{semiNM} may not even be defined at all, since every matrix $V \in \partial F(y)$ is singular; see \Cref{sys_sing} below.
Note that this now includes the important cases where the optimal
solution is a permutation matrix, a matrix that \emph{highly} disconnected.
We show that we can move from each iterate $y$ to a point
$y^\prime$ in the same equivalence class,
see~\Cref{def:Equivalence} below, so that we can
find a matrix that is \emph{non}singular
in the generalized Jacobian at $y^\prime$.

We now modify the semismooth Newton method \cref{semiNM} so that
the  iterates in the modified algorithm are well-defined, and the
convergence rate is quadratic even if $X^*$ is \emph{disconnected}.
The main idea for constructing well-defined iterates  is outlined
as follows: 
\begin{quote}
	for any vector $y$, construct an \emph{equivalent} vector $y^\prime$ so
	that there exists at least one nonsingular matrix in $\partial
	F(y^\prime)$ to obtain a well-defined next iterate.
\end{quote}

\subsection{Equivalence Classes}
This section introduces the notion of \emph{equivalence classes} of $y$ 
corresponding to a given dual feasible $X$. 
This is related to the \textdef{normal cone} at $X$.
Our Newton method finds iterates $y$, but we see
below that we are in particular interested in moving between equivalence classes
of $y$. And in particular, we are interested in a special point $y$ in
each equivalence class.

\subsubsection{Preliminaries}

We first define an equivalence relation for a partition of the underlying 
space $\R^{2n-1}$ to use for our modified Newton method.

\index{$y\sim y^\prime$, equivalent}	
\index{equivalent, $y\sim y^\prime$}

\index{$[y]$, equivalence class}
\begin{definition}[\textdef{equivalence class, $[y]$}]
	\label{def:Equivalence}
	Two vectors $y$ and $y^\prime$ in $\R^{2n-1}$ are equivalent, denoted by $y\sim y^\prime$, if 
	\[ 
	(\hat{x}+A^{T}y)_{+} = (\hat{x}+A^{T}y')_{+}.
	\]
	The set of equivalent vectors in $\R^{2n-1}$ is called the \textdef{equivalence class}. We denote the equivalence class to which $y$ belongs to by 
	$$[y]:=\{ y' \in \R^{2n-1} \;|\; y \sim y' \} . $$
\end{definition}

Recall that the nonnegative polar cone of a closed convex set $C$ at 
$w \in C$ is given by $(C-w)^+ = \{v : (c - w)^T v  \geq 0, \; \forall c
\in C\}$. We can show that each equivalence class is actually a
polyhedron that can be viewed in the $y\in \R^{2n-1}$ space, or
equivalently in the $x\in \R^{n^2}$ space.
The associated linear equations and inequalities are given explicitly in the next result.
	
	\begin{lemma}\label{y_poly}
		Let $\tilde y \in \R^{2n-1}$ and $\tilde x = (\hat{x} +
A^{T}\tilde y)_{+}$.
Then the following are equivalent:
\begin{enumerate}
\item
$ {y} \in [\tilde y]$
\item
		$$\begin{array}{rllllll}
			(A^T y)_i &=&(\hat x - \tilde x)_i &\text{ if
}& \tilde x_i>0,\\
			(A^T y)_j &\leq &(\hat{x} - \tilde x)_j &\text{
if }
& \tilde x_j=0.
		\end{array}$$
\item
\[
\tilde x - \hat x - A^T y  \in (\R_+^{n^{2}}-\tilde x)^+.
\]
\end{enumerate}
	\end{lemma}
	\begin{proof}
		A vector $y$ is contained in $[\tilde y]$ if, and only
if, $\tilde x$ is the optimal solution of the following  optimization problem 
		\begin{equation}\label{unique_opt2}
			\tilde x = \argmin_x \left\{ \frac 12 \|x-\hat
x-A^T {y}\|^2  :  x \in \R_+^{n^{2}}\right\}.
		\end{equation}
It follows from the classical Rockafellar-Pshenichnyi optimality
condition for \cref{unique_opt2}, that $\tilde x$ is an optimal solution
if, and only if, the gradient of the objective function at $\tilde{x}$ satisfies
		$$\tilde x - \hat x - A^T {y}  \in (\R_+^{n^{2}}-\tilde x)^+.$$
		This yields the third item. The second item follows from the fact that a vector $v \in (\R_{+}^{n^{2}}-\tilde x)^+$ is equivalent to
$v_{i} = 0$ for $\tilde x_{i} > 0$ and $v_{i} \geq 0$ for $\tilde x_{i} = 0$.
	\end{proof}

We now introduce some notation in order to facilitate the discussions 
about the disconnected case.
Let $y\in \R^{2n-1}$ and $X = \Mat\left(\hat{x}+A^{T}y\right)_{+} \in \R^{n\times n}$. Suppose that $X$ is disconnected with the following block diagonal structure:

%



\index{$X^{k}$, $k$-th diagonal block}
\index{$X^{ij}$, $(i,j)$-th off-diagonal block}
\index{$K$, blocks in $X$}
\index{blocks in $X$, $K$}

\begin{equation}\label{Xblk}
	X = \Blkdiag(X^{1},\ldots,X^{K}),
\end{equation}
where $X^{i} \in \R^{m_{i}\times n_{i}}$ is connected for
all $i=1,\ldots,K$. We write $y= 
\begin{pmatrix}  c\cr r\end{pmatrix}  
\in \R^{2n-1}$ correspondingly with the labels
$$\begin{array}{lll}
	c =  \begin{pmatrix} c^{1}\cr \vdots\cr c^{K}\end{pmatrix}\in \R^{n}, 
	&\text{ with }c^{i} \in \R^{n_{i}}, & \text{ for } i=1,\ldots,K,\\
	r = \begin{pmatrix} r^{1}\cr \vdots\cr r^{K}\end{pmatrix} \in \R^{n-1},
	&\text{ with } r^{i} \in \R^{m_{i}}, &\text{ for } i=1,\ldots,K-1, \text{ and } r^{K} \in \R^{m_{K}-1}.
\end{array}$$

The partition and its relation with $c^{i}$ and $r^{i}$ can be visualized as
\begin{equation}\label{yrc}
	X = \bordermatrix{& (c^{1})^T & \cdots & (c^{K})^T \cr
		r^{1}	& X^{1} &	   \cdots & X^{1,K}=0	\cr
		\vdots	& \vdots &	\ddots	  & \vdots	\cr
		r^{K}	& X^{K,1}=0 &	 \cdots & X^{K}	\cr  },
\end{equation}
where the off-diagonal blocks $X^{ij}$ $(i\neq j)$ are zero due to the disconnectedness assumption. Each diagonal block $X^{i}$ may be viewed as a smaller doubly stochastic matrix, if it is a square matrix. This motivates us to define the vectors by pairing $c^i$ and $r^i$:
\begin{equation}\label{ycr2}
	\begin{array}{lrlrll}
		\mathscr{Y}^{i} &=& \begin{pmatrix}
			c^{i} \\
			r^{i}
		\end{pmatrix} &\in& \R^{m_{i}+n_{i}} &\text{ for } i=1,\ldots,K-1, \vspace{.1cm} \\
		\mathscr{Y}^{K} &=& \begin{pmatrix}
			c^{K} \\
			r^{K}
		\end{pmatrix} &\in& \R^{m_{K}+n_{K}-1}. &
	\end{array}
\end{equation}
We use calligraphic letter $\mathscr{Y}^{i}$ to distinguish it from the $i$-th iterate $y^{i}$ in the Newton method \cref{semiNM} and \Cref{mo}.

We note that each diagonal block $X^{k}$ is completely determined by the vector $\mathscr{Y}^{k}$, i.e.,
\[
X_{ij}^k 
= \left( \hat{X}_{ij}^k  + c_j^k + r_i^k\right)_+
= \left( \hat{X}_{ij}^k + \mathscr{Y}^k_j + \mathscr{Y}^k_{n_k+i}  \right)_+ .
\]			
We also note that if two vectors $y$ and $\tilde{y}$ are equivalent, then the corresponding matrices $X = \Mat\left(\hat{x}+A^{T}y\right)_{+}$ and $\tilde{X} = \Mat\left(\hat{x}+A^{T}\tilde{y}\right)_{+}$ admit the same partition \cref{yrc}. 
Therefore, using the equivalence relation defined in \Cref{def:Equivalence}, it is unambiguous to speak of the  $(i,j)$-th off-diagonal block $X^{ij}$ or $(i,i)$-th diagonal block $X^{i}$ when it comes to the same equivalence class. 

\index{$\mathscr{Y}^i =(c^{i},r^{i})$}				
\index{$Y^{ij}$}
\index{$M^{ij}$}

Given $y\in \R^{2n-1}$, we list the notations to remind readers;
\begin{equation}\label{notations}
	\text{Each $y$ gives rise to }  \ 
	\begin{cases}
		Y=Y_y = \Mat\left(\hat{x}+A^{T}y\right) ,  \\
		X=X_y = \Mat\left(\hat{x}+A^{T}y\right)_+ ,\\
		M \in \cM(y),\\ 
		\cV(M) = A \Diag (\vec(M)) A^T \in \partial F(y),
	\end{cases}
\end{equation}
where we ignore the subscripts when the meaning is clear.
We partition the matrices $Y$ and $M$ in the same way as $X$ in \cref{yrc}, respectively. Denote by $Y^{ij}$ and $M^{ij}$ the $(i,j)$-th block of $Y$ and $M$, respectively. It is worthwhile to note that the off-diagonal blocks $Y^{ij}$ ($i\neq j$) are always \emph{non-positive} due to the block-diagonal structure of $X$.

This notation is extended verbatim to any other vectors in $\R^{2n-1}$. 
For example, if $\tilde{y} \in \R^{2n-1}$, then the symbols $\tilde{Y}$ and $\tilde{Y}^{ij}$ are unambiguously defined just as for $y$ above. In what follows, we will use these notations directly without defining them again.

\subsubsection{Uniqueness}	

\index{$U$}						
\index{row index set for block $X^k$, $R_{k}$}		
\index{column index set for block $X^k$, $C_{k}$}		
\index{$R_{k}$, row index set for block $X^k$}		
\index{$C_{k}$, column index set for block $X^k$}		

In this section we present sufficient conditions for the equivalence
class to be a singleton. We first note that uniqueness of the optimum $X^*$ means that
the solution set of the system \cref{nm_eq} is an equivalence class.

\begin{lemma}
	The solution set $ \{ y \;|\; F(y) = 0\}$ of the system \cref{nm_eq} is an equivalence class.
\end{lemma}
\begin{proof}
	The proof follows by definition,
	from the fact that the optimum $X^*$ is unique.
\end{proof}

{Although the optimal solution of the primal problem \cref{eq:mainprob} is unique, 
	the solution of the optimality conditions in \cref{eq:optcondG} for the dual
	variable $y$ is a compact, convex, nonempty set, 
	but is \emph{not} necessarily a singleton set in general. 
	The next result implies that we obtain a unique solution to
	\cref{eq:optcondG} when the unique primal optimal solution to
	\eqref{dsm:main} is \textdef{connected}.}

\index{unique!entry $y_i$}	
\index{unique!set}	
\begin{theorem}
	\label{unique_alg}
	Let $\tilde{y} \in \R^{2n-1}$ be given.	If $\Mat\left(\hat x+A^T\tilde{y}\right)_+$ is connected, then the equivalence class $[\tilde{y}]$ is a singleton.
\end{theorem}
\begin{proof}
	Recall that the equivalence class can be defined by the linear
equations and inequalities in \Cref{y_poly}. Applying the first proof in
\Cref{Vsing}, if $X$ is connected, then the columns of $A$ associated
with $x_{i} > 0$ form a basis of $\R^{2n-1}$. Therefore, the equations
in \Cref{y_poly} determine a unique solution. This implies that $[\tilde
y]$ is a singleton.
	
	We also provide an elementary proof below for the sake of self-containment.
Assume $X$ is connected, and let
$y = \begin{pmatrix} c\cr r\end{pmatrix} \in \R^{2n-1}$
be an element in $[\tilde{y}]$. The entry $y_{i}$ is said to be unique, if $\{y_{i} \;|\; y \in [\tilde{y}]\}$ has exactly one element. The subsets $R,C \subseteq \{1,\ldots,n\}$ are called unique, if the entries $r_{i}$ for $i \in R \backslash \{n\}$ and the entries $c_{j}$ for $j \in C$ are unique. We show that there exist unique subsets $R,C \subset \{1,\ldots,n\}$ and they can be extended so that $R = C =\{1,\ldots,n\}$.
\begin{enumerate}
	\item The existence: Let $R = \{n\}$ and $C \subseteq \{1,\ldots,n\}$ be such
	that $j \in C$ if, and only if, $X_{n,j} >0$. Since $X$ is connected, $C$ cannot be an empty set. As $X_{n,j} > 0$ for every $j \in C$, we have $X_{n,j} = (\hat{X}_{n,j} + c_{j})_{+} = \hat{X}_{n,j} + c_{j}$, see \cref{Xrc}. Thus the entries $c_{j}$ for $j \in C$ are uniquely determined. This shows that the subsets $R$ and $C$ are unique.
	
	\item The extension: Let the subsets $R,C \subseteq \{1,\ldots,n\}$ be unique.
	Since $X$ is connected, there exists at least one non-zero entry in
	$X_{\bar{R},C}$ or $X_{R,\bar{C}}$, see the paragraph after
	\Cref{blk_X}. Assume that $X_{\bar{R},C}$ contains a non-zero entry. Let
	$i \in \bar{R}$ be the row index associated with this non-zero entry. Then $X_{i,j} > 0$ for some $j \in C$, and this yields
	$$X_{ij} = (\hat{X}_{ij} + r_{i} + c_{j})_{+} = \hat{X}_{ij} + r_{i} + c_{j}.$$
	As $C$ is unique, $c_{j}$ is unique as $j \in C$, and thus $r_{i}$ is also unique. It follows that the subsets $R_{+} = R \cup \{i\}$ and $C_{+} = C$ are unique. The case when $X_{R,\bar{C}}$ contains some non-zero entries is similar.
\end{enumerate}
Therefore, $R = C = \{1,\ldots,n\}$ are unique, and this shows that $[\tilde{y}]$ has a unique solution.
\end{proof}

Note that \Cref{unique_alg} does not assume that $X$ is a doubly stochastic matrix. The uniqueness of the solution to the system \cref{eq:optcondG} follows directly from \Cref{unique_alg} as a special case when $X$ is a doubly stochastic matrix.

\begin{corollary}\label{unique}
	If the optimal solution $X^*$ of \cref{eq:mainprob} is connected, then 
	the solution $y^*$ to the system \cref{eq:optcondG} is unique.
	\qed
\end{corollary}

	\begin{remark}
		The converse direction in \Cref{unique_alg} doesn't hold. Assume that $X$ and $\hat{X}$ are both $2$ by $2$ identity matrices. Then $[y]$ contains vectors satisfying the system
		$$\begin{bmatrix}
			1 \\
			0 \\
			0 \\
			1
		\end{bmatrix} = \left(\begin{bmatrix}
			1 \\
			0 \\
			0 \\
			1
		\end{bmatrix} + \begin{bmatrix}
			1 & 0 & 1 \\
			1 & 0 & 0 \\
			0 & 1 & 1 \\
			0 & 1 & 0 \\
		\end{bmatrix}y\right)_{+} \text{ with variable } y \in \R^{3}.$$
		This system can be written equivalently as
		$$\begin{array}{ccc}
			\begin{array}{rrl}
				y_{1} + y_{3} &=& 0,\\
				y_{1} &\leq& 0, \\
				y_{2} + y_{3} &\leq& 0, \\
				y_{2} &=& 0. 
			\end{array}
		\end{array}$$
		We can easily derive that $y_{1} = y_{2} = y_{3} = 0$. Thus, there is a unique solution $y$ to the system even $X$ is disconnected.
	\end{remark}

\subsubsection{Polyhedron Description}	
The polyhedron characterization in \Cref{y_poly} does not exploit the structures in $A$. In this section, we provide a different characterization using the blocks in a disconnected $X$. This alternative characterization enables us to find a vertex of $[y]$ efficiently and prove the convergence of our algorithm.

Consider the vectors $c^{k}$ and $r^{k}$ associated with the $k$-th
diagonal block $X^{k}, k=1,\ldots, K$. If we add a constant to $c^{k}$ and subtract the same constant from $r^{k}$, then the diagonal block $X^{k}$ remains the same. We define a matrix $U$ associated with this operation as follows. Let $R_{k}$ and $C_{k}$ be the row and column indices corresponding to the $k$-th diagonal block of $X$, respectively. Define the matrix
\begin{equation}\label{shiftV}
	U = \begin{bmatrix}u^{1} & \cdots & u^{K}\end{bmatrix} \in \R^{2n-1 \times K},
\end{equation}
where the non-zero elements in each column $u^{k} \in \R^{2n-1}$ is given by
\begin{equation}\label{shiftv}
	\begin{array}{lrrl}
		u_{i}^{k} &=& -1 & \text{ for } i \in C_{k}, \\
		u_{n+i}^{k} &=&  1 & \text{ for } i \in R_{k}.
	\end{array}
\end{equation}				
It is clear that the aforementioned operation can be described by $y + \lambda_{k}u^{k}$, for some $\lambda_{k}\in \R$.

In \Cref{yV} below, we show that equivalent vectors in $[y]$ are in
the span of the first $K-1$ columns of the matrix $U$.

\begin{lemma}\label{yV}
	If $y \sim \tilde{y}$, then 
	\[
	\tilde{y} - y \in \range (\begin{bmatrix}u^{1} & \cdots & u^{K-1}\end{bmatrix}),
	\]
	the span of the first $K-1$ columns of $U$.
\end{lemma} 
\begin{proof}
	The statement in \Cref{unique_alg} can be extended trivially to non-square
	matrices. Note that each diagonal block $X^{k}$ is connected. If we fix an element in $r^{k} \in \R^{m_{k}}$, say the last entry $r_{m_{k}}^{k}$, then we can use the same argument in \Cref{unique_alg}  to see that all other entries in $\mathscr{Y}^{k} = \begin{pmatrix}
		c^{k} \\
		r^{k}
	\end{pmatrix}$ are uniquely determined. From this, we can deduce that $c^{k} = \tilde{c}^{k} + \lambda_{k}e$ and $r^{k} = \tilde{r}^{k} - \lambda_{k}e$ for some constant $\lambda_{k}$ for every $k=1,\ldots,K-1$. 
	
	Since the last block $X^{K}$ is connected, we can apply \Cref{unique_alg} to $X^{K}$. This shows that $\tilde{\mathscr{Y}}^{K} = \mathscr{Y}^{K}$ and thus $\lambda_{K} = 0$. Putting together, this implies that $\tilde{y} = y + U\lambda$ for some $\lambda \in \R^{K}$ with $\lambda_{K} = 0$ using the definition of $U$.
\end{proof}

The following result is a direct consequence of \Cref{yV}. It states that the associated diagonal blocks of $Y$ and $\tilde{Y}$ remain the same for the equivalent vectors $y,\tilde{y}$.
\begin{corollary}\label{Yblk}
	If $y \sim \tilde{y}$, then $\mathscr{Y}^{K} = \tilde{\mathscr{Y}}^{K}$ and $Y^{k} = \tilde{Y}^{k}$ for $k=1,\ldots,K$.
\end{corollary}


We now show that every equivalence class has a polyhedral representation via $U$. 

\begin{theorem}\label{equi}
	Let $\tilde{y} \in \R^{2n-1}$. The equivalence class $[\tilde{y}]$ is a polyhedron given by
	\begin{equation}
		\label{eq:equi_poly}
		[\tilde{y}] = \{ y \in \R^{2n-1} \;|\;  y = \tilde{y} + U\lambda \text{ for some } \lambda \in \R^{K}, \lambda_{K} = 0 \text{ and } Y^{ij} \leq 0 \text{ for } i\neq j\}, \footnote{The redundant variable $\lambda_{K}$ in \cref{eq:equi_poly} is included to simplify the proof in \Cref{thm_close0}.}
	\end{equation}					
	where $Y^{ij}$ is the $(i,j)$-the block of $Y = \Mat(\hat x+A^Ty)$ 
	with respect to the partition associated with $\tilde{y}$ as defined in \cref{Xblk}.
\end{theorem}
\begin{proof}
	Let $y$ be a vector on the right hand side set from \eqref{eq:equi_poly}.
	By the definition of $U$, the matrices $X$ and $\tilde{X}$ have the same diagonal blocks. Since $X = (Y)_{+}$ and the off-diagonal blocks $Y^{ij} \leq 0$ are non-positive, the off-diagonal blocks $X^{ij} = (Y^{ij})_{+} = 0$. This shows that $X= \tilde{X}$ and thus $y \in [\tilde{y}]$.  
	
	Conversely, for any vector $y \in [\tilde{y}]$, we have $y = \tilde{y} + U\lambda$ and $\lambda_{K} = 0$ by \Cref{yV}. Since $y \in [\tilde{y}]$,  we must have $Y^{ij} \leq 0$ for $i\neq j$.
	Therefore, $y$ is contained in the set on the right-hand-side. 
\end{proof}

\subsubsection{Vertices}	

\index{maximal!$\bar{s}$}				
\index{maximal!$\Mat(\bar{s})$}				
\index{maximal!$\cV(\bar{s})$}				

For any equivalence class $[\tilde{y}]$, we aim to find a vector $y \in
[\tilde{y}]$ so that $\partial F(y)$ contains at least one nonsingular
matrix. For $y \in \R^{2n-1}$,  the matrix $M \in \cM(y)$ is said to be \emph{maximal}, if 
\begin{equation}\label{maxs}
	(\hat{x}+A^{T}y)_{ij} \geq 0 \Longrightarrow M_{ij} =  1.
\end{equation}
For a maximal $M'$, it is easy to see that $M' \geq M$ and
thus $\cV(M') \succeq \cV(M)$ for every $M \in \cM(y)$. Therefore, if
$\partial F(y)$ contains a nonsingular matrix, then $\cV(M')$ must
be nonsingular. In this case, we also call the matrices $\cV(M') \in \partial F(y)$ \emph{maximal}.

It turns out that the generalized Jacobian at any vertex of the polyhedron $[\tilde{y}]$ contains at least one nonsingular matrix.

\begin{theorem}\label{vertices}
	Let $\tilde y\in \R^{2n-1}$ be given, $y \in [\tilde{y}]$, and let
	$M$ be maximal for $y$. The vector $y$ is a vertex of the
	polyhedron $[\tilde{y}]$ if, and only if, $M$ is connected.
\end{theorem}
\begin{proof}
	Assume $M$ is disconnected. Without loss of generality, we can write
	$$M = \begin{bmatrix}
		M^{1} & 0 \\
		0 & M^{2}
	\end{bmatrix} \text{ and } Y = \begin{bmatrix}
		Y^{1} & Y^{12} \\
		Y^{21} & Y^{2}
	\end{bmatrix}.$$
	It holds that $Y^{12} < 0$ and $Y^{21} < 0$ by the maximality of $M$. Thus, there exists an $\epsilon > 0$ such that the vectors $y^\prime = y + \epsilon u^{1}$ and $y^{\prime\prime} = y - \epsilon u^{1}$ are in $[\tilde{y}]$, where the vector $u^{1}$ is defined as in \cref{shiftv}. But then $y = \frac{1}{2}y^\prime + \frac{1}{2}y^{\prime\prime}$ and thus $y$ is not a vertex.
	
	Conversely, assume that $M$ is connected. If $X$ is connected, then \Cref{unique} implies that the polyhedron $[\tilde{y}] = y$ and thus $y$ is an extreme point. Therefore, we assume that $X$ is disconnected and consider its partition as given in \cref{yrc}. Suppose for the sake of contradiction that $y \in [\tilde{y}]$ is not an extreme point. Then there exist a scalar $\alpha \in (0,1)$ and vectors $y^\prime,y^{\prime\prime} \in [\tilde{y}]$ both different from $y$ such that 				
	$y = \alpha y^\prime + (1-\alpha) y^{\prime\prime}$. Then, $Y =  \alpha Y^\prime + (1-\alpha) Y^{\prime\prime}$. Note that the diagonal blocks of $Y,Y^\prime$ and $Y^{\prime\prime}$ are the same, see \Cref{Yblk}.

	Since $M$ is connected, one of the off-diagonal blocks $M^{i,K}, M^{K,i}$ for $i=1,\ldots,K-1$ must contain a positive entry. By the maximality of $M$, we have that $M^{ij}>0$ if, and only if, $Y^{ij} \geq 0$. This implies that one of the off-diagonal blocks $Y^{i,K}, Y^{K,i}$ for $i=1,\ldots,K-1$ must contain a nonnegative entry, say the $(i,j)$-th entry $Y_{ij}^{K-1,K}$ of the $(K-1,K)$-th block. In addition, as $X$ is disconnected and $X = (Y)_{+}$, the off-diagonal blocks $Y^{i,K}, Y^{K,i}$ for $i=1,\ldots,K-1$ must be non-positive. Putting together, the entry $Y_{ij}^{K-1,K}$ must be zero.

	Similarly, we have that  $(Y^\prime)^{K-1,K} \leq 0$ and $(Y^{\prime\prime})^{K-1,K} \leq 0$, and therefore, the equation 
	$$0 = Y_{i,j}^{K-1,K} = \alpha (Y^\prime)_{i,j}^{K-1,K} + (1-\alpha) (Y^{\prime\prime})_{i,j}^{K-1,K}$$ implies that  $(Y^\prime)_{ij}^{K-1,K} = (Y^{\prime\prime})_{ij}^{K-1,K}= 0$. Therefore, it holds that (see \cref{notations,Xrc})
	$$\begin{array}{rrrrl}
		0 &=& (Y^\prime)_{ij}^{K-1,K} &=&  \hat{X}_{ij}^{K-1,K} + (r^\prime)_{i}^{K-1} + (c^\prime)_{j}^{K}, \vspace{.2cm}\\
		0 &=& (Y^{\prime\prime})_{ij}^{K-1,K} &=&  \hat{X}_{ij}^{K-1,K} + (r^{\prime\prime})_{i}^{K-1} + (c^{\prime\prime})_{j}^{K}.
	\end{array}$$
	From \Cref{Yblk}, we know that $(\mathscr{Y}^{\prime})^{K} = (\mathscr{Y}^{\prime\prime})^{K}$ and thus $(c^\prime)_{j}^{K} = (c^{\prime\prime})_{j}^{K}$. This implies that $(r^\prime)_{i}^{K-1} = (r^{\prime\prime})_{i}^{K-1} $. It then follows from \Cref{equi} that $(\mathscr{Y}^{\prime})^{K-1} = (\mathscr{Y}^{\prime\prime})^{K-1}$.  This argument can be repeated for all the remaining diagonal blocks until we get $y = y^\prime = y^{\prime\prime}$. This yields contradiction, and thus $y$ is an extreme point.
\end{proof}

It follows from \Cref{Vsing} and \Cref{vertices} that $\partial F(y)$ contains a nonsingular matrix whenever $y$ is a vertex of the polyhedron $[\tilde{y}]$. More precisely, the maximal $\mathcal{V}(M)$ is nonsingular when $y$ is a vertex. This result is stated in the next corollary.

\begin{corollary}\label{nonsingular}
	If $y$ is a vertex of the polyhedron $[\tilde{y}]$, then $\partial F(y)$ contains at least one nonsingular matrix. In particular, the maximal matrix $V \in \partial F(y)$ is nonsingular.
\end{corollary}

The rest of this section provides a method for finding a vertex efficiently.  We start with the existence of a vertex for any polyhedron $[y]$.
\begin{lemma}\label{existy}
	Let $y \in \R^{2n-1}$. The polyhedron $[y] \subset \R^{2n-1}$ contains at least one vertex.
\end{lemma}
\begin{proof}
	A polyhedron contains a line if there exists a vector $y \in \R^{2n-1}$ and a non-zero direction $d\in \R^{2n-1}$ such that $y+\alpha d$ is contained in the polyhedron for all scalars $\alpha$. It is well known that a polyhedron has at least one vertex if, and only if, it does not contain a line.
	
	If $X$ is connected, then the polyhedron $[y]$ contains only one vector $y$ which is a vertex by \Cref{unique}. Suppose that $X$ is disconnected. Assume, without loss of generality, that we can write $X$ and $Y$ as
	$$X = \begin{bmatrix}
		X^{1} & 0\\
		0 & X^{2}
	\end{bmatrix} \text{ and } Y = \begin{bmatrix}
		Y^{1} & Y^{12} \\
		Y^{21} & Y^{2}
	\end{bmatrix},$$
	where the diagonal block $X^{2}$ are connected. Here, the block $X^{1}$ does not have to be connected. 
	
	Let $\tilde{y} := y + \alpha d$ for some $d$ and define
	$$D := \Mat\left(A^{T}d\right)= \begin{bmatrix}
		D^{1} & D^{12} \\
		D^{21} & D^{2}
	\end{bmatrix}.$$
	It follows from \Cref{Yblk} that the entries in $y$ and $\tilde{y}$ associated with the last connected block are the same, i.e., $\mathscr{Y}^{2}= \mathscr{\tilde{Y}}^{2}$. Thus, the entries in $d$ associated with $D^{2}$ must be zero. From this, we can see that if the direction $d$ is non-zero, then there exists at least one non-zero element in $D^{12}$ or $D^{21}$.
	
	If $y \sim \tilde{y}$, then applying \Cref{Yblk} again yields $Y_{2} = \tilde{Y}_{2}$ and this implies that
	$$\tilde{Y} = \begin{bmatrix}
		\tilde{Y}^{1} & \tilde{Y}^{12} \\
		\tilde{Y}^{21} & \tilde{Y}^{2} \\							
	\end{bmatrix} = \begin{bmatrix}
		Y^{1}+ \alpha D^{1} & Y^{12}+ \alpha D^{12} \\
		Y^{21}+ \alpha D^{21} & Y^{2} \\							
	\end{bmatrix}. $$
	In both cases, at least one of the entries in the off-diagonal blocks $\tilde{Y}^{12}$ or $\tilde{Y}^{21}$ becomes positive for sufficiently large or small $\alpha$. This shows that $\tilde{y}$ is not equivalent to $y$ for all $\alpha$. Thus $[y]$ doesn't contain a line, and it has at least one vertex.
\end{proof}

	\begin{remark}\label{polytope}
		In \Cref{existy}, if we assume additionally that $X = \Mat\left(\hat{x}+A^{T}y\right)_{+}$ does not contain any zero rows or columns, then $[y]$ is even bounded and thus a polytope.  We prove this by contradiction. Assume that $[y]$ is not bounded. Then there exists a non-zero direction $d \in \R^{2n-1}$ such that $y + \alpha d \in [y]$ for all $\alpha \geq 0$. By \Cref{equi}, we have that $d = U\lambda$ for some non-zero $\lambda \in \R^{K}$ with $\lambda_{K} = 0$. Thus, $\tilde{y}:=y + \alpha U \lambda \in [y]$ for all $\alpha \geq 0$. The $(i,j)$-th off-diagonal blocks of $Y$ and $\tilde{Y}$ satisfy $$\tilde{Y}^{i,j} = Y^{i,j} + (\lambda_{i} - \lambda_{j})J,$$ where $J$ is all-ones matrix of appropriate size. As $\lambda_{K} = 0$ and $\lambda \neq 0$, there exists an index $i \in \{1,\ldots,n-1\}$ such that $\lambda_{i} - \lambda_{K} > 0$ or $\lambda_{K} - \lambda_{i} > 0$. This implies that the blocks $\tilde{Y}^{i,K}$ or $\tilde{Y}^{K,i}$ contain a positive entry for sufficiently large $\alpha$. But then $y$ and $\tilde{y}$ are equivalent. This is a contradiction. Therefore, $[y]$ is always bounded.
		
		Finally, the problem \cref{eq:mainprob} satisfies Mangasarian-Fromovitz constraint qualification and this implies the set of dual optimal solutions is bounded. This yields an alternative derivation that the optimal set $[y^{*}]$ is bounded.
	\end{remark} 

We can find a vertex of the polyhedron $[\tilde{y}]$ as follows. In
\Cref{equi}, $[\tilde{y}]$ is expressed as the projection of a higher
dimensional polyhedron in variables $y\in \R^{2n-1}$ and $\lambda \in
\R^{K-1}$. Through the Fourier–Motzkin elimination, we can describe
the polyhedron $[\tilde{y}]$ solely using variables $y\in \R^{2n-1}$.
Then a vertex of $[\tilde{y}]$ can be obtained via solving a particular linear program. This procedure, however, is very expensive. In what follows, we provide an efficient combinatorial method for finding a vertex of $[\tilde{y}]$.

\index{$X^{\bar{R},\bar{C}}$}
\index{$X^{R,C}$}

\index{$J$, all-ones matrix}
\index{all-ones matrix, $J$}

\begin{lemma}\label{thm:shift}
	Let $y\in \R^{2n-1}$. Let $X$ be disconnected 						and
	$$X = \begin{bmatrix}
		X^{R,C} & 0\\
		0 & X^{\bar{R},\bar{C}}\\
	\end{bmatrix} \quad \text{ and } \quad Y =  \begin{bmatrix}
		Y^{R,C} & Y^{R,\bar{C}}\\
		Y^{\bar{R},C} & Y^{\bar{R},\bar{C}}\\
	\end{bmatrix},$$
	for some subsets $R,C$ as in \cref{blk_X}.
	Let $\tilde{y} := y + t u^{1}$ for some $t \in \R$, where $u^{1}$ is defined as in \cref{shiftv} for the partition above. Then	
	\begin{equation}\label{thm:shifteq}
		y \sim \tilde{y} \quad \Longleftrightarrow \quad  \max_{i,j} Y_{i,j}^{\bar{R},C} \leq t \leq  - \max_{i,j}Y_{i,j}^{R,\bar{C}}.
	\end{equation}
	
\end{lemma}
\begin{proof}
	The scalar $t$ is well-defined, as $ Y^{\bar{R},C}$ and $	Y^{R,\bar{C}}$ are non-positive. Then the matrix $\tilde{Y}$ can be written as
	\begin{equation}\label{Yform}
		\tilde{Y} = \begin{bmatrix}
			Y^{R,C} & Y^{R,\bar{C}} + t J\\
			Y^{\bar{R},C} - tJ & Y^{\bar{R},\bar{C}}\\
		\end{bmatrix},
	\end{equation}
	where $J$ is the all-ones matrix of appropriate sizes.
	We see that $y\sim \tilde{y}$ if, and only if, $Y^{R,\bar{C}} + t J \leq 0$ and $Y^{\bar{R},C} - t J \leq 0$. The latter is equivalent to the inequalities in \cref{thm:shifteq}.
\end{proof}

\index{$J$, all-ones matrix}				
\index{all-ones matrix, $J$}

For any vector $y$, we can find a vertex of $[y]$ efficiently.
\begin{theorem}\label{ysing}
	For any vector $y \in \R^{2n-1}$, there is a polynomial-time algorithm for finding a vertex of the polyhedron $[y]$.
\end{theorem}
\begin{proof}
	Let $X,Y$ and the maximal matrix $M$ defined as in \cref{notations,maxs} associated with $y$. Denote by $M^{i,j}$ the $(i,j)$-th block of $M$ $(i,j=1,\ldots,K)$ corresponding to the partition of $X$ in \cref{Xblk}. If $y$ is not a vertex, then $M$ is disconnected. Thus, there exists a subset $\mathcal{B}\subseteq \{1,\ldots,K\}$ such that $K \in \mathcal{B}$,
	$$Y^{R,\bar{C}} < 0 \text{ and } Y^{\bar{R},C} < 0,$$
	where $R$ and $C$ be the collection of row and column indices of $Y$ associated with the blocks in $\mathcal{B}$. For example, if $\mathcal{B} = \{K\}$, then
	$$Y^{R,\bar{C}} = \begin{bmatrix}
		Y^{1,K} \\
		\vdots \\
		Y^{K-1,K}
	\end{bmatrix} < 0 \text{ and } Y^{\bar{R},C} = \begin{bmatrix}
		Y^{K,1} & \cdots & Y^{K,K-1}
	\end{bmatrix} < 0.$$
	This means we can find a constant $t \neq 0$ such that $\max Y^{\bar{R},C}\leq t \leq -\max Y^{R,\bar{C}}$ as in \cref{thm:shifteq}. Let $\tilde{y} = y + tw$, where $w\in \R^{2n-1}$ is defined as
	\begin{equation}\label{shiftw}
		\begin{array}{lrrl}
			w_{i} &=& -1 & \text{ for } i \in C, \\
			w_{n+i} &=&  1 & \text{ for } i \in R.
		\end{array}
	\end{equation}
	
	By \Cref{thm:shift}, we have $\tilde{y}\sim y$. Recall that $\tilde{Y}$ has the form \cref{Yform}. In particular, we distinguish the following two cases depending on $t$:
	\begin{equation}\label{rule}
		\begin{array}{ll}
			1. & \text{If we take $t = -\max Y^{R,\bar{C}} > 0$, then $Y^{R,\bar{C}}+t J$ contains at least one zero entry.} \\[.2cm]
			2. & \text{Similarly, if $t = \max  Y^{\bar{R},C} < 0$, then $Y^{\bar{R},C}-t J$ contains at least one zero entry.}
		\end{array}
	\end{equation}
	Let $\tilde{M}$ be the maximal matrix defined similarly for $\tilde{y}$. In either case, the number of non-zero elements in $\tilde{M}$ is strictly less than these in $M$. As $y\sim \tilde{y}$, we can repeat this procedure until $M$ is connected.
\end{proof}

\subsection{The Algorithm and its Local Convergence}

For any vector $y\in \R^{2n-1}$, we can find a vertex $\tilde{y}$ of the polyhedron $[y]$ using \Cref{ysing}. It follows from \Cref{nonsingular} that the maximal matrix $\tilde{V}\in \partial F(\tilde{y})$ is nonsingular. 
Thus we can generate well-defined iterates when maximal $\tilde{V} \in \partial F(\tilde{y})$ is used at each iteration. We achieve this by developing a variant of the Semismooth Newton method. 


\begin{algorithm}[H]
	\begin{algorithmic}[1]
		\STATE \textbf{Require:} $y^0$ initial point, $tol$ tolerance
		\WHILE{$\|F(y^k)\| > tol$}
		\STATE Find a vertex $\tilde{y}^{k}$ of $[y^{k}]$ using \Cref{ysing}
		\STATE Compute the maximal $\tilde{V}_{k} \in \partial F(\tilde{y}^{k})$
		\STATE Update $y^{k+1} = \tilde{y}^{k} - \tilde{V}_{k}^{-1}F(\tilde{y}^{k})$ \ENDWHILE
	\end{algorithmic}
	\caption{A Modified Semismooth Newton Method}
	\label{mo}
\end{algorithm}

We now prove Q-quadratic local convergence of the modified Newton method.
Recall that the distance between a vector $y \in \R^{2n-1}$ to a subset $S
\subseteq \R^{2n-1}$ is
\begin{equation}
	\label{eq:ptsetdist}
	\begin{array}{rrl}
		\textdef{$\dist(y,S)$} : = \inf_{s \in S}\|y-s\|.
	\end{array}
\end{equation}
Similarly, we denote the nearest point
distance between two subsets $S,T \subset \R^{2n-1}$ as
\begin{equation}
	\label{eq:setsetdist}
	\textdef{$\dist(S,T)$} : = \inf_{s \in S, t \in T}\|s-t\|.
\end{equation}

The main idea behind the proof
is that if an equivalence class $[y]$ is sufficiently
close to the optimal set $[y^{*}]$ in the sense of \eqref{eq:setsetdist},
then every element in $[y]$ is also close to $[y^{*}]$ in the sense of
\eqref{eq:ptsetdist}; and this further implies that
each vertex in $[y]$ is also close to one of the vertices in $[y^{*}]$. 
For any polyhedron $[y]$, we denote by $\ext[y]$ the set of vertices of $[y]$.
\index{$\ext[y]$, vertices of $[y]$}
\index{vertices of $[y]$, $\ext[y]$}

\begin{lemma}\label{thm_close0}
	Suppose that $F(y^{*}) = 0$. Then there exist $\epsilon >0$ and $\kappa
	>0$ such that for any $y \in \R^{2n-1}$ with $\dist(y,[y^{*}]) <
	\epsilon$ we have:
	\begin{enumerate}
		\item $\dist(\tilde{y},[y^{*}]) < \kappa \cdot \epsilon$ for every $\tilde{y} \in [y]$.
		\item $\dist(\tilde{y},\ext[y^{*}]) < \kappa \cdot \epsilon$ for every $\tilde{y} \in \ext [y]$.
	\end{enumerate}
\end{lemma}
\begin{proof}
	\begin{enumerate}
		\item Let $y \in \R^{2n-1}$. Without loss of generality, we assume that $y^{*}$ satisfies $\|y - y^{*}\| = \dist(y,[y^{*}]) < \epsilon$. It follows from \Cref{pattern} that if $\epsilon > 0$ is sufficiently small, then
		\begin{equation}\label{XXstar}
			X_{ij}^* > 0 \quad 
			\Longrightarrow \quad X_{ij} > 0.
		\end{equation}
		Let $X$ and $X^{*}$ be partitioned as in \Cref{Xblk},
		\begin{equation}
			\label{eq_par}
			\begin{array}{rrl}
				X     & = & \Blkdiag\left(X^{1},\ldots,X^{K}\right),                 \\
				X^{*} & = & \Blkdiag\left((X^{*})^{1},\ldots,(X^{*})^{K^{*}}\right),
			\end{array}
		\end{equation}
		where $K$ and $K^{*}$ are the number of blocks in $X$ and $X^{*}$, respectively. It follows from \Cref{XXstar} that $K \leq K^{*}$, and  moreover, we can view each block $(X^{*})^{i,j}$  as a unique sub-block of $X^{k,l}$ for some $k,l = 1,\ldots,K$. As an example, assume we have the following partition for $X$ and $X^{*}$ into $K=2$ and $K^*=3$ blocks, respectively,
		$$X = 
		\kbordermatrix{	
			&	& & & & & \\
			& X_{1,1} & X_{1,2} &  X_{1,3} & \vrule & 0 & 0 \\
			& X_{2,1} & X_{2,2} &  X_{2,3} & \vrule & 0 & 0  \\ 
			& X_{3,1} & X_{3,2} &  X_{3,3} & \vrule & 0 & 0  \\ \cline{2-7}			
			& 0 & 0 & 0 & \vrule  &  X_{4,4} & X_{4,5} \\		
			& 0 & 0 & 0 & \vrule  & X_{5,4} &  X_{5,5} \\		
		}, X^{*} = 
		\kbordermatrix{	
			&	& & & & & \\
			& X_{1,1}^{*} & X_{1,2}^{*} & \vrule &  0 & \vrule & 0 & 0 \\
			& X_{2,1}^{*} & X_{2,2}^{*} & \vrule &  0 & \vrule & 0 & 0  \\ \cline{2-8}			
			& 0 & 0 & \vrule &  X_{3,3}^{*} & \vrule & 0 & 0  \\ \cline{2-8}			
			& 0 & 0 & \vrule & 0 & \vrule  &  X_{4,4}^{*} & X_{4,5}^{*} \\		
			& 0 & 0 & \vrule  & 0 & \vrule  & X_{5,4}^{*} &  X_{5,5}^{*} \\		
		}.$$
		Then the top-left block $(X^{*})^{1} \in \R^{2 \times 2}$ and the mid block $(X^{*})^{2} \in \R^{1}$ of $X^{*}$ on the right hand side are both sub-blocks of $X^{1} \in \R^{3 \times 3}$ of $X$ on the left hand side. 
		
		 For convenience, we define a zero-one matrix $P \in \{0,1\}^{K^{*} \times K}$ such that $P_{ij} = 1$ if, and only if, $(X^{*})^{i}$ is a sub-block of $X^{j}$. For instance, the matrix $P$ in the previous example is given by
		 $$ P = \begin{bmatrix}
		 	1 & 0 \\
		 	1 & 0 \\
		 	0 & 1
		 \end{bmatrix} \in \R^{3 \times 2}.$$
		  Denote by $p_{i} \in \R^{K}$ the $i$-th row of $P$. Note that $p_{i} = p_{j}$ if, and only if, the diagonal blocks $(X^{*})^{i}$ and $(X^{*})^{j}$ are the sub-blocks of the same diagonal block in $X$. We also define the matrices $U$ and $U^{*}$ as in \cref{shiftV} for $X$ and $X^{*}$, respectively.

		From \Cref{equi}, we know that if $\tilde{y} \in [y]$, then $\tilde{y} = y + U \tilde{\lambda}$ for some $\tilde{\lambda} \in \R^{K}$ with $\tilde{\lambda}_{K} = 0$.	Define $y' := y^{*} + U\tilde{\lambda}$. By construction,  the distance between $\tilde{y}$ and $y'$ is small as \begin{equation}\label{eq_close1}
			\|\tilde{y} - y^\prime \| = \|y +U\tilde{\lambda} - y^{*} - U\tilde{\lambda}\| = \|y - y^{*}\|< \epsilon.
		\end{equation}	
		We will show that $\dist(y^\prime ,[y^{*}])$ is also sufficiently small. The key idea is that $y'$ only slightly violates the set of constraints defining the polyhedron $[y^{*}]$ in \Cref{equi}. From this, we can establish an upper bound for $\dist(y^\prime ,[y^{*}])$ using Hoffman's error bound \cite{hoffman}. Together with \cref{eq_close1}, the first inequality in the statement follows immediately.
		
		Let $\lambda^{*} = P \tilde{\lambda} \in \R^{K^*}$. Since the last diagonal block $(X^{*})^{K^{*}}$ of $X^{*}$ must be a sub-block of $X^{K}$ and $\tilde{\lambda}_{K} = 0$, we have the equality
		\begin{equation}\label{eq_close2}
			\lambda_{K^{*}}^{*} = 0.
		\end{equation} 
		One can easily verify that $U = U^{*}P$ and thus $U\tilde{\lambda} = U^{*}\lambda^{*}$. This means we can write \begin{equation}\label{eq_close3}
			y^\prime :=  y^{*} + U\tilde{\lambda} = y^{*} + U^{*}\lambda^{*}.
		\end{equation}
		From \Cref{equi}, the equivalence class $[y^{*}]$ can be defined as
		\begin{equation}\label{ystar2}
			[y^{*}] = \left\{ y^{*} + U^{*}\lambda \in \R^{2n-1} \;|\;  \begin{array}{cc}
				\lambda \in \R^{K^{*}}, \lambda_{K^{*}} = 0\\ 
				(Y^*)^{ij} \leq 0 \text{ for } i\neq j
			\end{array}\right\},
		\end{equation}
		where $(Y^{*})^{ij}$ is the $(i,j)$-th block of $Y^{*} = \Mat\left(\hat{x} +A^{T}y^{*}\right)$ with respect to the partition of $X^{*}$ in \cref{eq_par}. From \cref{eq_close2} and \cref{eq_close3}, it is clear that $y'$ satisfies the equality constraints in \cref{ystar2}.

		
		As to the inequality constraints, we partition $Y^\prime$  in the same way as $X^{*}$. Each block $(Y^\prime)^{ij}$ can also be viewed as a sub-block of $X^{k,l}$ for some $k,l=1,\ldots,K$. We distinguish the following two cases based on the off-diagonal blocks in $Y^\prime$.
		\begin{enumerate}
			\item

			If $(Y^\prime)^{ij}$ $(i\neq j)$ is a sub-block of a diagonal block $X^{k}$, then the diagonal blocks $(Y^\prime)^{i}$ and $(Y^\prime)^{j}$ are also sub-blocks of $X^{k}$. As $\lambda^*=P\tilde{\lambda}$, this means $\lambda_{i}^{*} = \lambda_{j}^{*} =\tilde{\lambda}_{k}$. Since $y^\prime = y^* + U\tilde{\lambda} = U^{*}\lambda^{*}$, we obtain that $(Y')^{ij} = (Y^{*})^{ij} \leq 0$.
			
			\item  If $(Y^\prime)^{ij}$ is a sub-block of an off-diagonal block $Y^{k,l}$, then the constraint
			$(Y^\prime)^{ij} \leq 0$ may not be satisfied. As $\|\tilde{y} - y^\prime\| < \epsilon$ from \cref{eq_close1}, it holds that
			$$\|\tilde{Y}-Y^\prime\| = \|\hat{x} +A^{T}\tilde{y} - \hat{x} - A^{T}y^\prime\| = \|A^{T}(\tilde{y}-y^\prime)\|< \|A\|\epsilon.$$
			In addition, $\tilde{y} \in [y]$ implies that
$\tilde{Y}^{k,l} \leq 0$, see \Cref{equi}. This means the largest nonnegative entry in $(Y^\prime)^{i,j}$ is at most $\|A\|\epsilon$.
		\end{enumerate}
		This shows that $y^\prime$ violates the constraints in $[y^{*}]$ only up to a scalar multiplication of $\epsilon$. The Hoffman's error bound implies that $\dist(y^\prime,[y^{*}]) < c \cdot \epsilon$ for some universal constant $c$ which depends only on the matrix $A$ and the polyhedron $[y^{*}]$. 
		
		We can establish the first inequality now. Let $v \in [y^{*}]$ such that $ \|y^\prime-v\| = \dist(y^\prime,[y^{*}]) < \epsilon$. For any $\tilde{y} \in [y]$, we have that
		$$\begin{array}{rrl}
			\dist(\tilde{y},[y^{*}]) &\leq& \|\tilde{y} - v\|  \\
			&\leq& \|\tilde{y} - y^\prime\| + \|v - y^\prime\|  \\
			&=& \|\tilde{y} - y^\prime \| + \dist(y^\prime,[y^{*}])\\
			&<& c_{1}\epsilon,
		\end{array}$$
		where $c_{1} = c + 1$.
		
		\item For any $\tilde{y} \in \ext [y]$, we know from the first part that $\dist(\tilde{y},[y^{*}]) < c_{1} \cdot \epsilon$ for some universal constant $c_{1}$. Without loss of generality, we assume that $y^{*}$ satisfies $\|\tilde{y} - y^{*}\| = \dist(\tilde{y},[y^{*}])$. If $y^{*} \in \ext [y^{*}]$, then $\dist(\tilde{y},\ext [y^{*}]) < c_{1} \cdot \epsilon$. Thus, we assume that $y^{*} \notin \ext [y^{*}]$.
		
		We transform $y^{*}$ into a vertex $\tilde{y}^{*} \in \ext [y^{*}]$ using the procedure in the proof of \Cref{ysing}. The obtained vertex $\tilde{y}^{*}$ depends on the choice in \cref{rule}. This yields a sequence $\lambda_{i}^{*}$ such that $\tilde{y}^{*} = y^{*} + \sum_{i=1}^{m}\lambda_{i}^{*}w_{i}$, where $w_{i}$ is defined as in \cref{rule} and $m$ is the number of iterations. Note that $m \leq K^{*}-1$.  In what follows, we show that it is possible to pick a sufficiently small $\lambda_{i}^{*}$ at each iteration.
		
		In the first iteration, we identify subsets $R$ and $C$ such that $(Y^{*})^{\bar{R},C}<0$ and $(Y^{*})^{R,\bar{C}}<0$. Then we choose either $\lambda_{1}^{*} = \max (Y^{*})^{\bar{R},C}$ or $\lambda_{1}^{*} =  -\max (Y^{*})^{R,\bar{C}}$. As $\|\tilde{y}- y^{*}\| = \dist(\tilde{y},[y^{*}]) < c_{1} \cdot \epsilon$, it holds that 
		\begin{equation}\label{eq_close4}
			\|\tilde{Y}-Y^{*}\| = \|A^{T}(\tilde{y}-y^{*})\|< c_{1}\|A\|\epsilon = \epsilon_{1},
		\end{equation}
		where we set  $\epsilon_{1} := c_{1}\|A\|\epsilon$.
		Therefore, we obtain that 
		\begin{equation}\label{eq_close5}
			\|\tilde{Y}^{\bar{R},C} - (Y^{*})^{\bar{R},C}\| < \epsilon_{1} \text{ and } \| \tilde{Y}^{R,\bar{C}} - (Y^{*})^{R,\bar{C}}\| < \epsilon_{1}.
		\end{equation}  Since $\tilde{y} \in \ext[y]$ is a vertex, the associated maximal matrix $\tilde{M}$ is connected by \Cref{vertices}, see the definition of maximality in \cref{maxs}. This implies that 
		\begin{equation}\label{eq_close6}
			\max \tilde{Y}^{\bar{R},C} \geq 0 \text{ or } \max \tilde{Y}^{R,\bar{C}} \geq 0,
		\end{equation}
		as otherwise $\tilde{M}$ is disconnected.
		Using \cref{eq_close5,eq_close6}, we conclude that $$\max (Y^{*})^{\bar{R},C} > -\epsilon_{1} \text{ or }  \max (Y^{*})^{R,\bar{C}} > -\epsilon_{1}.$$
		We choose $\lambda_{1}^{*}$ as follows.
		\begin{enumerate}
			\item 		If $\max (Y^{*})^{\bar{R},C} > -\epsilon_{1}$, then $\lambda_{1}^{*} = \max (Y^{*})^{\bar{R},C}$.
			\item		If $\max (Y^{*})^{R,\bar{C}}  > -\epsilon_{1}$, then $\lambda_{1}^{*} =  -\max (Y^{*})^{R,\bar{C}}$.
		\end{enumerate}
		As $\lambda_{1}^{*} < 0$, we have that $|\lambda_{1}^{*}| < \epsilon_{1}$ in both cases. In the second iteration, we apply the same procedure to $y^{*} + \lambda_{1}^{*}w_{1}$. The same argument above can be used, except that $\epsilon_{1}$ is replaced by $2\epsilon_{1}$, to show that that $|\lambda_{2}^{*}| < 2\epsilon_{1}$. Proceeding in this way, we conclude that $|\lambda_{k}^{*}| < 2^{k}\epsilon$ for $k=1,\ldots,m$.
		
		These upper bounds for $\lambda_{1}^{*},\ldots,\lambda_{m}^{*}$ imply that
		$$\begin{array}{rrl}
			\|\tilde{y} - \tilde{y}^{*}\| &=& \| (\tilde{y}-y^{*}) - \sum_{k=1}^{m}\lambda_{k}^{*}w^{k}\| \\
			&\leq&\|\tilde{y}-y^{*} \| +  \|\sum_{k=1}^{m}\lambda_{k}^{*} w^{k}\| \\
			&\leq& c_{1} \cdot \epsilon + \sum_{k=1}^{m} |\lambda_{k}^{*}| \cdot \|w^{k}\| \\
			&\leq& c_{1} \cdot \epsilon +  (\max_{k} \|w^{k}\|) (\sum_{k=1}^{m} 2^{k})\epsilon_{1}  \\
			&<& c_{2} \cdot \epsilon, 
		\end{array}$$
		for some constant $c_{2}$ depending on $m$. As $\tilde{y}^{*} \in \ext[y^{*}]$, we have that  $\dist(\tilde{y}, \ext [y^{*}]) \leq \|\tilde{y} - \tilde{y}^{*}\| < c_{2} \cdot \epsilon$. 
	\end{enumerate}
	Finally, we take $\kappa = \max \{c_{1},c_{2}\}$ and this finishes the proof.
\end{proof}



%

We provide the convergence of the modified Newton method.

\begin{theorem}\label{thmmain}
	Let the current iterate $y^k$ be sufficiently close to the (compact, 
	convex) solution set $[y^*]$.
	Then the modified Newton method converges, and at a Q-quadratic rate, to a 
	point in $[y^*]$.
\end{theorem}
\begin{proof}
	For any $y \in \R^{2n-1}$, if $M \in \mathcal{M}(y)$ is maximal in \cref{My}, then $M$ is an $n$ by $n$
	zero-one matrix, see also \cref{maxs}. This means that
	there are at most $2^{n^{2}}$ different maximal matrix in $\mathcal{M}(y)$. If $\bar{\mathcal{M}}$ is the collection of different maximal
	matrices in $\partial F(y)$, i.e.,
	$$\bar{\mathcal{M}}:=\{ M \;|\; M \in \mathcal{M}(y) \text{ is maximal}\},$$
	then $|\mathcal{V}|$ is finite. Therefore, there exists a constant $\beta$ such that $\|V^{-1}\| \leq \beta$ for every $ V \in \bar{\mathcal{M}}$.

	\index{semismooth}
	Let $K^{*}$ be the number of blocks in the unique optimal solution $X^*$. Let $0 < \eta < \min\{1 ,\frac{1}{\beta \kappa^{2}}\}$ be a fixed constant, where $\kappa$ is the constant in \Cref{thm_close0}. Since $F$ is semismooth at any optimal solution $y^{*}$, there exists $\epsilon > 0$ such that 
	\begin{equation}\label{local}
		\|F(y^{*})-F(y)- V\left(y^{*} - y\right)\| \leq \eta\|y^{*} -
		y\|^{2}, \quad \forall y \in B(y^{*},\epsilon) \text{ and } V \in
		\partial F(y),
	\end{equation}
	where \textdef{$B(y^{*},\epsilon)$} is the $\epsilon$ ball around $y^*$.
	Recall that the number of vertices of any polytope is finite, and $\bar{\mathcal{M}}$ is a finite set. Thus, for any fixed $\eta$, we can assume that the above inequality \cref{local} holds for every vertex $\tilde{y}^{*}$ of $[y^{*}]$. 
	
	Let $\tilde{y}^{k}$ be any vertex of $[y^{k}]$ obtained from \Cref{ysing} in the algorithm. Let $\tilde{y}^{*} \in \ext[y^{*}]$ be such that $ \|\tilde{y}^{k} - \tilde{y}^{*}\| = \dist(\tilde{y}^{k},\ext [y^{*}])$. It holds that
	\[
	\begin{array}{rrl}
		\dist(y^{k+1},[y^{*}]) & \leq & \|y^{k+1} - \tilde{y}^{*}\| \\
		&=& \|\tilde{y}^{k} - \tilde{y}^{*} - \tilde{V}_{k}^{-1}F(\tilde{y}^{k})\|\\
		&=& \|\tilde{V}_{k}^{-1} ( F(\tilde{y}^{*}) - F(\tilde{y}^{k}) - \tilde{V}_{k}(\tilde{y}^{*} - \tilde{y}^{k}) ) \| \\
		&\leq& \|\tilde{V}_{k}^{-1}\| \cdot \| F(\tilde{y}^{*}) - F(\tilde{y}^{k}) -
		\tilde{V}_{k}(\tilde{y}^{*} - \tilde{y}^{k})  \|  \\
		&\leq & \beta \eta \|\tilde{y}^{k} - \tilde{y}^{*}\|^{2},
	\end{array}
	\]		
	where the last inequality follows from~\cref{local}.
	If $\dist(y^{k},[y^{*}]) = \epsilon > 0$ is sufficiently small, then \Cref{thm_close0} shows that that $$\|\tilde{y}^{k} - \tilde{y}^{*}\|^{2} = \dist(\tilde{y}^{k},\ext [y^{*}])^{2} < \kappa^{2} \cdot \epsilon^{2} = \kappa^2 \cdot \dist(y^{k},[y^{*}])^{2}.$$ Thus, this yields
	$$\begin{array}{rrl}
		\dist(y^{k+1},[y^{*}]) &\leq & \beta \eta \|\tilde{y}^{k} - \tilde{y}^{*}\|^{2}  \\
		&\leq & \beta \eta  \kappa^{2} \dist(y^{k},[y^{*}])^{2} \\
		&<& \dist(y^{k},[y^{*}])^{2}.
	\end{array}$$
	
	This shows that $\dist(y^{k+1},[y^{*}]) < \dist(y^{k},[y^{*}])^{2}$, and thus the modified Newton method converges quadratically to the optimal set $[y^{*}]$. 
\end{proof}

We observe that the performance of \Cref{mo} depends on the number of blocks $K^{*}$ in the optimal solution $X^*$. In \Cref{thm_close0}, the constant $\kappa$ depends on $K^{*}$. If $K^{*}$ is large, then the condition for the quadratic convergence in \Cref{thmmain} is stricter. This suggests that an instance can be more difficult to solve if the optimal solution $X^{*}$ contains many blocks. Our numerical experiment verifies this observation, 
see \Cref{fig:blocks}.


Finally we discuss about an undesirable phenomenon called \textdef{cycling}. If $y^{k} = y^{k'}$ for some $k < k'$, then we say the algorithm is cycling. Thus, the algorithm may
loop indefinitely. Fortunately, if $y^{k}$ is sufficiently close to $[y^{*}]$ as required in \Cref{thmmain}, then cycling cannot happen as $\dist([y^{k+1}],[y^{*}]) < \dist([y^{k}],[y^{*}])^{2}$. 	In the general case, we can avoid cycling empirically by taking a random choice in the step \cref{rule} in \Cref{ysing}. This
generates a random vertex each time. With this simple trick, we never
end up in a cycle in our numerical experiments. Therefore we focus on
the case when cycling does not occur. (It is worth mentioning that this 
cycling is similar to the simplex method cycling for degenerate
problems, i.e.,~when the simplex algorithm remains stuck at the same feasible
vertex. However, unlike the simplex method, the total number of vertices in our problem is not finite.)


\index{cycling}

\section{Refinement and the Local Error Bound Condition}				
In this section we show that we can split the problem into smaller problems when the iterate $y$ in the (modified) Newton method is sufficiently close to the solution $y^{*}$ of \cref{eq:mainprob}. Under the strict complementarity assumption, 
we can split the problem recursively until the assumption in \Cref{main_conv} holds; we obtain the solutions for each subproblem  by the semismooth Newton method \cref{semiNM}.

Recall that if $y^{*}$ is a solution to the system \cref{nm_eq}, then $x^{*} = (\hat{x}+A^{T}y^{*})_{+}$ is an optimal solution to \cref{eq:mainprob} and $z^{*} = (\hat{x}+A^{T}y^{*})_{-}$ is an optimal dual variable for \cref{eq:mainprob}, see  \Cref{thm:optcondG}. We say that \textdef{strict complementarity} holds at $(x^{*},z^{*})$, if $x^{*} + z^{*} > 0$.


\begin{lemma}\label{pattern}
	Suppose $F(y^{*}) = 0$. There exists an $\epsilon >0$ such that for every $y$ satisfying $||y-y^{*}|| < \epsilon$, it holds that
	\begin{equation}\label{pos1}
		x_{i}^{*} > 0 \Longrightarrow x_{i} > 0,
	\end{equation}
	where $x = (\hat{x}+A^{T}y)_{+}$. Moreover, if  $(x^{*},z^{*})$ satisfies strict complementarity, then we can also take $\epsilon$ such that $$x_{i}^{*} > 0 \Longleftarrow x_{i} > 0.$$
	
\end{lemma}
\begin{proof}
	Let $A_i$ denote the $i$-th column of $A$. 
	If $x_{i}^{*} > 0$, then $x_{i}^{*} = \hat{x}_{i}+A_{i}^{T}y^{*} > 0$
	and thus $x_{i} = (\hat{x}_{i}+A_{i}^{T}y)_{+} = \hat{x}_{i}+A_{i}^{T}y
	> 0$ for small $\epsilon > 0$.	Now suppose that the pair $(x^*,z^*)$
	satisfies strict complementarity. Assume to the contrary that $x_{i}^{*} = 0$. Then we have $z_{i}^{*} > 0$, i.e., $\hat{x}_{i}+A_{i}^{T}y^{*} < 0$, and thus  $\hat{x}_{i}+A_{i}^{T}y < 0$ for sufficiently small $\epsilon$. It follows that $x_{i} = (\hat{x}_{i}+A_{i}^{T}y)_{+} = 0$.
\end{proof}

Suppose $y$ is close to $y^{*}$. \Cref{pattern} suggests that the $X = \Mat (\hat{x}+A^{T}y)_{+}$ and the optimal solution $X^{*}=\Mat (\hat{x}+A^{T}y^*)_{+}$ share the same block-diagonal structure. 
As a heuristic, we can split the problem into smaller subproblems if the residual is sufficiently small. If strict complementarity holds, then the smaller subproblems will not be disconnected eventually and thus the semismooth Newton method \cref{semiNM} can be applied.

The local error bound condition is a sufficient condition for the convergence of Newton-type methods. It is a weaker requirement than the nonsingularity (i.e., connectedness) condition used in \Cref{sec:semi}. In this section, we show that the system \cref{nm_eq} for the nearest doubly stochastic matrix problem does not satisfy the local error bound condition.

%
%
%
%

\index{$[y^*]$, solution set}
\index{solution set, $[y^*]$}

\begin{definition}[{\textdef{local error bound}}]
	\label{def:localerrorbnd}
	Let $[y^*]$ be the solution set of \cref{nm_eq} and let $N$ be a
	neighbourhood such that $[y^*] \cap N \neq \emptyset$.
	If there exists a positive constant $c$ such that 
	\begin{equation}\label{localerror}
		c \cdot \dist(y,[y^*]) \leq ||F(y)||, \quad \forall 
		y \in N,
	\end{equation}
	then we say that $F$ satisfies the \emph{local error bound condition on
		$N$} for the system \cref{nm_eq}.
\end{definition}

We show that the local error condition does not hold for \cref{nm_eq}, and this implies that $\partial F(y)$ is singular in general. Recall that strict complementarity holds for \cref{nm_eq} if $x +z >0$ for optimal primal and dual variables $x$ and $z$.

\begin{theorem}\label{leb}
	Consider the system \cref{nm_eq}.
	Assume that strict complementarity holds. Then $F(y)$ in \cref{nm_eq}
	does not satisfy the local error bound condition.
\end{theorem}
\begin{proof}
	Let $y\in \R^{2n-1}$. Define the projection 
	\begin{equation}\label{leb_P}
		P_{C}(y):= \argmin_{u\in C} \|u - y\|, \text{ where $C$ is a polyhedron}.
	\end{equation}
	Let $y^{*} = P$ and $d = y - y^{*}$. Note that $\text{dist}(y,[y^{*}]) = ||y-y^{*}|| = ||d||$. 
	
	Let $x = (\hat{x} + A^Ty)_{+}$ and $x^{*} = (\hat{x} + A^Ty^{*})_{+}$. Define $s^{*} \in \{0,1\}^{n^2}$ such that $s_{i}^{*} = 1$ if, and only if, $x_{i}^{*} > 0$. Applying \cref{pos1} in \Cref{pattern}, we can assume $\|d\|$ is sufficiently small so that $x_{i} >0$ if, and only if, $x_{i}^{*} >0$. Therefore, it holds that
	$$\begin{array}{rrl}
		(\hat{x} + A^Ty)_+ &=& (\hat{x} + A^Ty^{*} +A^{T}d)_+\\
		&=& x^{*} + \Diag(s^{*})A^{T}d
	\end{array}$$
	Thus we have
	$$\begin{array}{rrl}
		||F(y)|| &=& ||A(\hat{x} + A^Ty)_+ - b|| \\
		&=& ||Ax^{*} + A\Diag(s)A^{T}d - b|| \\
		&=& ||A\Diag(s)A^{T}d||.
	\end{array}$$
	If $X^{*}$ is disconnected, then $A\Diag(s)A^{T} \in \partial F(y^{*})$ is singular. Let $\epsilon > 0$. Let $\{y^{i}\}$ be a sequence in $\R^{2n-1}$ such that $d^{i} = y^{i} - y^{*}$ with ,  and 
	
	Let $\{d^{i}\}$ be a sequence in $\R^{2n-1}$ such that
$\|d^{i}\| = \|d\|$. Assume that the sequence $\{d^{i}\}$ converges to a
vector in the null space of $A\Diag(s)A^{T}$. The normal fan of
$[y^{*}]$ is complete, see Definition 7.1 and Example 7.3 in
\cite{zie95forqap}. 	It follows from the classical Rockafellar-Pshenichnyi optimality condition for the minimization problem  \cref{leb_P} that there always exists a vector  $y^{i} \in \R^{2n-1}$ such that $d^{i} = y^{i} - P_{[y^{*}]}(y^{i})$ for each $d^{i}$. This yields
	$$\frac{||F(y^{i})||}{ \text{dist}(y^{i},[y^{*}])} = \frac{||A\Diag(s)A^{T}d^{i}||}{||d||} \rightarrow 0.$$
	This shows that there exists no positive constant $c$ such that \cref{localerror} holds, and thus the local error bound condition fails.
\end{proof}
\index{normal fan}

\section{Numerical Experiments}

In this section, we present numerical tests for
the modified Newton algorithm. The main purpose is to illustrate 
empirically the
correctness of our proposed algorithm. 
(Further extensive testing of these semismooth
methods are given in \cite{QiSun:06,bai2007computing}.)

First, we compare
\Cref{mo} with the standard interior point method (IPM), and the
alternating direction method of multipliers (ADMM).
For the interior point
method, the problem is modelled in CVX \cite{cvx} and then solved using
MOSEK solver \cite{aps2019mosek}. For ADMM, we transform \cref{eq:mainprob} 
to the equivalent problem $\min \left\{\frac 12 \|x-\hat x\|^2 \,:\,
x=y, Ay=b, \, x\geq 0\right\}$. The coupling constraint $x=y$ induces a
standard splitting in the polyhedral cone variable $x$ and the linear equality
variable $y$. For more details about the implementation of ADMM applied to the least square problems, we refer to \cite{he2011solving}.




For the numerical experiments, we generate the data $\hat{X}$ from the
standard normal distribution. 
Throughout \Cref{table1,table2}, $n$ refers to the size of $\hat{X} \in \Rnn$;
\emph{iteration} refers to the number of iterations; 
\emph{opt.cond.} refers to the sum of the norms of the optimality 
conditions
in~\cref{eq:optcondxyz_copy} at termination, i.e.,~primal and dual feasibility
and complementary slackness;
and \emph{time} refers to the total running time in seconds. 

\Cref{table1} displays the numerical results for one instance of
sizes $n=100,\ldots,500$, respectively. We compare the three methods. 
It is clear that the modified semismooth Newton method has a superior running 
time to ADMM and IPM. It also does better with respect to the optimality
conditions.
The tolerance for the optimality conditions for IPM is 
from the default obtained from MOSEK. As expected,
interior point methods have difficulty obtaining more than square root of machine
epsilon accuracy.
The accuracy for ADMM methods take significantly longer if more accuracy 
is requested.
In addition from~\cref{eq:optcondplusminus}, we see that both dual feasibility
and complementary slackness hold exactly for the NM algorithm, and the optimality
conditions error is totally from the primal feasibility residual
$\|Ax-b\|$.
\begin{table}[H]
	\centering{
		\begin{tabular}{|c|ccc|ccc|ccc|}
			\hline
			& \multicolumn{3}{c|}{The modified NM \Cref{mo}} &   \multicolumn{3}{c|}{IPM}    &   \multicolumn{3}{c|}{ADMM}   \\ \hline
			$n$  & iteration & opt. cond. &         time          & iteration & opt. cond. & time & iteration & opt. cond. & time \\ \hline
			100  & 9 & 1.2e-14 & 0.1 & 25 & 4.0e-10 &  0.51 & 941 & 9.9e-13 &  0.21 \\
			200  & 13 & 1.8e-14 & 0.1 & 26 & 1.4e-06 &   1.5 & 1735 & 9.9e-13 &   1.3 \\
			300  & 12 & 7.5e-15 & 0.18 & 22 & 6.8e-07 &   2.2 & 2746 & 1.0e-12 &   4.3 \\
			400  & 12 & 7.8e-15 & 0.33 & 22 & 1.3e-05 &   4.3 & 3834 & 1.0e-12 &    17 \\
			500  & 13 & 5.3e-15 & 0.55 & 25 & 4.9e-07 &   8.1 & 4634 & 1.0e-12 &    30 \\ \hline
		\end{tabular}
	}
	\caption{Small instances}
	\label{table1}
\end{table}

In \Cref{table2} below we present the numerical results of larger instances,
$n=1000,1500,2000$. We did not include results for IPM or ADMM as they took
significantly longer.
\begin{table}[H]
	\centering{
		\begin{tabular}{|c|ccc|} \hline
			& \multicolumn{3}{c|}{The modified NM \Cref{mo}} \\ \hline
			$n$ &  iteration & opt. cond. & time \\\hline
1000  & 11 & 1.4e-15 & 0.47  \\
2000  & 11 & 1.1e-15 & 1.6  \\
3000  & 12 & 6.8e-16 & 3.9  \\
4000  & 13 & 3.6e-16 & 7.6  \\
5000  & 13 & 4.3e-16 & 12  \\
			\hline 
		\end{tabular}
	}
	\caption{Medium and large instances}
	\label{table2}
\end{table}

Next, we compare \Cref{mo} with the semismooth Newton-CG algorithm (SSNCG1) presented in \cite{li2020efficient}. Roughly speaking, SSNCG1 avoids the singularity of the Jacobian matrix by adding a scaled identity matrix $\epsilon I$ at each iteration for some $\epsilon >0$. The scalar $\epsilon$ is determined by the residual at the current iteration and a number of parameters. The SSNCG1 also involves a line-search to determine its step length. We also note that the problem formulation in \cite{li2020efficient} does not remove the redundant constraint as \cref{eq:mainprob}. 

\index{SSNCG}


We generate test instances whose optimal solution $X^{*}$ has many blocks. As discussed in the paragraph after \Cref{thmmain}, we expect that an instance is difficult to solve if there are many blocks in $X^{*}$. This is substantiated in the numerical results in \Cref{fig:blocks}. 
	
In addition, we observe that \Cref{mo} consistently takes less iterations than SSNCG1. This may be explained by our fast quadratic convergence. However, the running time of \Cref{mo} is longer than SSNCG1 due to the costly vertex finding step in \Cref{mo}. More specifically, the update in \Cref{Yform} takes a significant amount of the running time. The update \Cref{Yform} can be computed much more efficiently if the algorithm is implemented in C.


\begin{figure}[H]
	\centering
	\includegraphics[scale=.6]{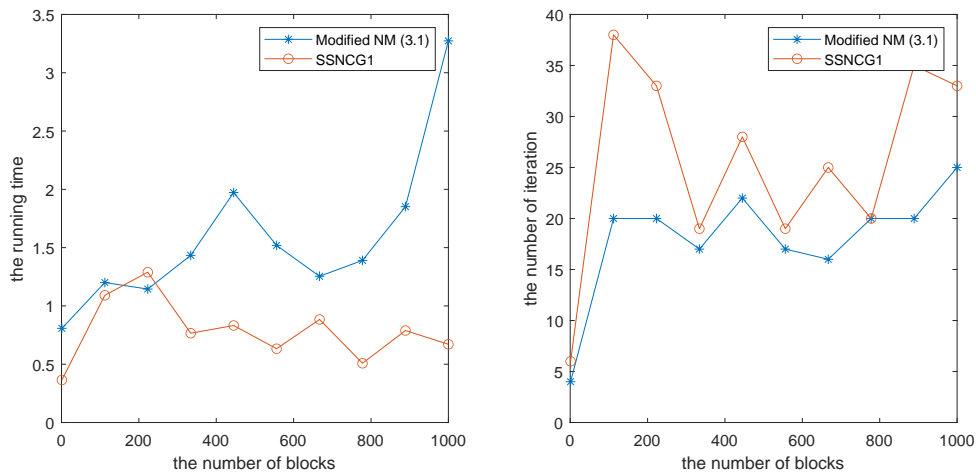}
	\caption{Problem instances of size $n=1000$, but with different number of blocks in the optimal solution $X^{*}$}
	\label{fig:blocks}
\end{figure}

\section{Conclusion}
The nearest doubly stochastic matrix problem is formulated as a system of
strongly semismooth equations. We show that this system does not satisfy the 
so-called local error bound condition, and therefore, the quadratic convergence 
of a Newton-type method may not be guaranteed. We exploit the problem structure 
to construct a modified Newton method that converges to the solution at
a quadratic rate. The novelty of the proposed algorithm is that the search space is partitioned into equivalence classes to overcome degeneracy.	
This partitioning strategy can be extended to more general problems.
This is also the first known Newton-type method which enjoys quadratic
convergence in the absence of the local error bound condition.

\section{Acknowledgement}
We would like to thank Xudong Li, Defeng Sun and Kim-Chuan Toh for sharing their codes in \cite{li2020efficient} with us.

\newpage
\printindex
\addcontentsline{toc}{section}{Index}
\label{ind:index}

\bibliographystyle{siam}
\bibliography{.master,.psd,.qap,.haoh,xinxin,.haesol}
\addcontentsline{toc}{section}{Bibliography}

\end{document}